\documentclass[12pt]{article}
\usepackage[utf8]{inputenc}
\usepackage[T1]{fontenc}

\usepackage{amsmath}
\usepackage{amsfonts}
\usepackage{latexsym}
\usepackage[]{graphicx}  
\usepackage{amssymb}
\usepackage{amsthm}
\usepackage{thmtools}
\usepackage{thm-restate}
\usepackage[margin=2cm]{geometry}
\usepackage{rotating}
\usepackage{url}
\usepackage{color}
\usepackage{enumerate}
\usepackage[shortlabels]{enumitem} 
\usepackage[small]{caption}
\usepackage[noadjust]{cite}
\usepackage{framed}
\usepackage{ifdraft}
\usepackage[affil-it]{authblk}
\usepackage{tikz}
\usetikzlibrary{arrows,shapes,automata,matrix,calc,plotmarks} 
\usepackage{relsize}
\tikzset{fontscale/.style = {font=\relsize{#1}}}
\usepackage{caption}
\usepackage{subcaption}

\ifdraft{
    \usepackage[colorinlistoftodos]{todonotes}
    \usepackage{letltxmacro}
    \LetLtxMacro{\oldtodo}{\todo}
    \usepackage{xspace}
    \renewcommand{\todo}[2][]{\oldtodo[#1]{#2}\xspace}%
    \usepackage{xcolor}
    \usepackage[mark,dirty={Modified!}]{gitinfo2}
    
}{
    \usepackage[disable]{todonotes}
}

\usepackage{pdftexcmds}
\makeatletter
\newcommand{\stringcases}[3]{%
  \romannumeral
    \str@case{#1}#2{#1}{#3}\q@stop
}
\newcommand{\str@case}[3]{%
  \ifnum\pdf@strcmp{\unexpanded{#1}}{\unexpanded{#2}}=\z@
    \expandafter\@firstoftwo
  \else
    \expandafter\@secondoftwo
  \fi
    {\str@case@end{#3}}
    {\str@case{#1}}%
}
\newcommand{\str@case@end}{}
\long\def\str@case@end#1#2\q@stop{\z@#1}
\makeatother

\newif\iflongversion
\longversiontrue    
\iflongversion
  
  \usepackage{comment}
  \excludecomment{extendedabstractversion}
\else
  
  \usepackage{comment}
  \excludecomment{longversion}
  \usepackage{environ}
  \NewEnviron{killcontents}{}

\fi

\usepackage[draft=false,breaklinks,backref=page,colorlinks]{hyperref} 

\renewcommand*{\backref}[1]{}
\renewcommand*{\backrefalt}[4]{%
   \ifcase #1 %
     \relax %
   \or
     (page #4).%
   \else
     (pages #4).%
   \fi%
}
\usepackage[nameinlink]{cleveref}

\newtheorem{definition}{Definition}
\newtheorem{lemma}[definition]{Lemma}
\newtheorem{proposition}[definition]{Proposition}
\newtheorem{corollary}[definition]{Corollary}

\newtheorem*{conjecture*}{Conjecture}
\newtheorem{theorem}[definition]{Theorem}
\theoremstyle{remark}
\newtheorem{example}[definition]{Example}
\newtheorem{remark}[definition]{Remark}

\newcommand{\N}{\mathbb{N}}
\newcommand{\Z}{\mathbb{Z}}
\newcommand{\Q}{\mathbb{Q}}
\newcommand{\R}{\mathbb{R}}

\newcommand{\A}{\mathcal A}
\newcommand{\B}{\mathcal B}
\newcommand{\C}{\mathcal C}

\newcommand{\bv}{{\bf v}}

\newcommand{\sigmaC}{{\sigma_{\mathrm{c}}}}

\DeclareMathOperator\per{per}

\makeatletter
\newcommand\xtwoheadrightarrow[2][]{%
  \ext@arrow 9999{\longtwoheadrightarrowfill@}{#1}{#2}}
\newcommand\longtwoheadrightarrowfill@{%
  \arrowfill@\relbar\relbar\twoheadrightarrow}
\makeatother
\newcommand{\Tedge}[4]{#1 \xrightarrow{\left. #2 \middle| #3 \right.} #4}
\newcommand{\Twalk}[4]{#1 \xtwoheadrightarrow{\left. #2 \middle| #3 \right.} #4}
\newcommand{\assoc}[1]{\widetilde{#1}}

\newcommand{\mat}[4]{\begin{pmatrix}
#1 & #2 \\ #3 & #4
\end{pmatrix}}
\newcommand{\mata}{\mat{a}{b}{c}{d}}
\newcommand{\D}[1][n]{\mathcal{D}_{#1}}
\newcommand{\Dj}{\D[1]{}}
\newcommand{\DB}[1][n]{\mathcal{DB}_{#1}}
\newcommand{\RB}[1][n]{\mathcal{RB}_{#1}}
\newcommand{\CB}[1][n]{\mathcal{CB}_{#1}}
\newcommand{\LS}[1][n]{\mathcal{LS}_{#1}}
\newcommand{\RS}[1][n]{\mathcal{RS}_{#1}}
\newcommand{\LE}[1][n]{\mathcal{LE}_{#1}}
\newcommand{\RE}[1][n]{\mathcal{RE}_{#1}}
\newcommand{\T}[1][n]{\mathcal{T}_{#1}}

\newcommand{\LclassMax}[1]{\nu_L \left ( #1 \right )}

\newcommand{\RclassMax}[1]{\nu_R \left ( #1 \right )}

\crefname{theorem}{theorem}{theorems}
\crefname{corollary}{corollary}{corollaries}
\crefname{example}{example}{examples}
\crefname{lemma}{lemma}{lemmas}
\crefname{proposition}{proposition}{propositions}
\crefname{definition}{definition}{definitions}
\crefname{observation}{observation}{observations}

\begin{document}

\title{Bounds on the period of the continued fraction after a M\"obius transformation}


\author{Hanka Řada}
\affil{\footnotesize Faculty of Nuclear Sciences and Physical Engineering\\ Czech Technical University in Prague\\ Czech Republic}
\author{Štěpán Starosta%
\thanks{Electronic address: \texttt{stepan.starosta@fit.cvut.cz}}}
\affil{\footnotesize Faculty of Information Technology\\ Czech Technical University in Prague\\ Czech Republic}

\date{}

\maketitle

\begin{abstract}
We study M\"obius transformations (also known as linear fractional transformations) of quadratic numbers.
We construct explicit upper and lower bounds on the period of the continued fraction expansion of a transformed number as a function of the period of the continued fraction expansion of the original number.
We provide examples that show that the bound is sharp.
\end{abstract}


\ifdraft{
   \listoftodos
}

\section{Introduction}
Eventually periodic continued fraction expansions correspond exactly to quadratic irrational numbers.
Some general upper bounds on periods of such an expansion, depending on the number itself, are known, see \cite{Podsypanin,Pohl07}.
In some very specific cases, the exact value is known (and so is the expansion), see \cite{cohn1977length,Rockett-1990,Vladimi-2009,BaHru}.

We study such periods after a transformation which preserves eventual periodicity of the expansion.
Given a nonsingular matrix $N = \begin{pmatrix} a & b \\ c & d \end{pmatrix} \in \Z^{2,2}$, we consider the mapping $h_N: \R \setminus \left\{  - \frac{d}{c} \right\} \to \R$ given by
\[
h_N(x) = \frac{ax + b}{cx + d}.
\]
Such a mapping is called \emph{the M\"obius transformation associated with the matrix N}, it is also sometimes referred to as a linear fractional transformation.
Given a quadratic irrational number $x$, the number $h_N(x)$ is clearly a quadratic irrational number (in the same field).
Our main result is an upper and lower bound on the period of the continued fraction expansion of $h_N(x)$ as a function of the period of the continued fraction expansion of $x$.

To state the main result, we introduce the following notation and definitions.
We have $x = [v, \overline{w}]$ with $v \in \N^\ell$ and $w \in \N^k$ for some $\ell, k \in \N, k \neq 0$ (overline denotes infinite repetition of $w$).
If such a sequence $w$ is the shortest possible, we say that it is the \emph{repetend} of the continued fraction expansion of $x$ (or simply of $x$).
Let $\per(x)$ denote the shortest \emph{period} of a continued fraction of $x$, i.e. the length of the repetend of $x$.

For nonnegative integers $a,c$ not both zero, let $\xi(a,c)$ denote the number of divisions required to compute the $\gcd(a,c)$ using the Euclidean algorithm (ending when $0$ is reached).
Thus, for instance, $\xi(7,0) = \xi(0,7) = 0$, $\xi(7,1) = \xi(1,7) = 1$, and $\xi(13,5) = \xi(5,13) = \xi(5,3) + 1 = \xi(3,2) + 2 = \xi(2,1) + 3 = 4$.

Our main result are the following bounds on $\per(h_N(x))$.

\begin{theorem} \label{thm:main}
Let $x$ be a quadratic irrational number, $h_N$ a M\"obius transformation and $n = |\det N|$.
We have
\[
\frac{1}{S_n} \per(x) \leq \per(h_N(x)) \leq S_n \per(x),
\]
where
\[
S_n = \sum_{\substack{t \in \N \\ t \mid n}} \sum_{\substack{j = t\\ j \not \in J_t }}^{2t-1} \left (2\left\lfloor \frac{\xi(j,t)}{2} \right\rfloor +1 \right)
\]
with $J_t  = \begin{cases}  \emptyset & \text{for } \gcd(t,\frac{n}{t}) = 1, \\ \{i \gcd(t,\frac{n}{t}) \colon i \in \N\} & \text{otherwise.} \end{cases}$
\end{theorem}

\begin{example}
Let us give an example for $N = \begin{pmatrix}
12 & 1 \\ 17 & 2
\end{pmatrix}$.
We have $|\det N| = 7$ and $S_7 = 24$.
For $x_1 = [\overline{3}]$, we have $\per(x_1) = 1$ and $\per(h_N(x_1)) = 6 \leq S_7 \per(x_1) = 24$.
For $x_2 = [\overline{200}]$, we have $\per(x_2) = 1$ and $\per(h_N(x_2)) = 24 \leq S_7 \per(x_2) = 24$.
For 
\[
x_3 = [-1, 1, 11, \overline{7, 1, 6, 8, 399, 8, 6, 1, 7, 3, 2, 7, 1, 2, 1, 1, 7, 1, 1, 2, 1, 7, 2, 3}],
\]
we have $\per(x_3) = 24$ and $\per(h_N(x_3)) = 1 \geq \frac{\per(x_3)}{S_7} = 1$
\end{example}

The presented proof of \Cref{thm:main} is based on the famous work of Raney \cite{Raney1973} who described transducers which output the continued fraction expansion of $h_N(x)$ while inputting the continued fraction expansion of $x$.

Note that the action of M\"obius transformations on a number $x$ has been explored for instance in \cite{LaSh97} and \cite{Liu} where the authors give bounds on the value of partial coefficients of the number $x$ after transformation.

The proof of the main result may be considered somewhat technical.
In \Cref{sec:prelim}, we first give some necessary notations and recall results of Raney.
A number of additional claims and the proof of \Cref{thm:main} are in \Cref{sec:constr}.
The last section contains remarks and experiment results.

\section{Preliminaries} \label{sec:prelim}

\subsection{LR representation}

For a more detailed description of the computations with continued fractions, we need another representation of positive real numbers.
Before we state it, we introduce the following notation.

\emph{An alphabet} $\A$ is a finite set of symbols.
\emph{A word} over the alphabet $\A$ is a sequence of symbols from this alphabet.
If the sequence is empty, it is \emph{the empty word} and it is denoted by $\varepsilon$.
The set of all finite words over an alphabet $\A$ is denoted by $\A^*$ and the set of all finite and infinite words by $\A^{\N}$.
If we have $v,w \in \A^{\N}$, then $vw$ denotes the concatenation of the words $v$ and $w$.
If there exists $u \in \A^{\N}$ such that $w = vu$, we say that $v$ is \emph{a prefix} of $w$.
Moreover, if $w \neq v$, we say that $v$ is \emph{a proper prefix} of $w$.
Analogously, if $v,w \in \A^{\N}$ and there exists $u \in \A^*$ such that $w = uv$, we say that $v$ is \emph{a suffix} of $w$ and moreover, if $w \neq v$, we say that $v$ is \emph{a proper} suffix of $w$.
A word $u$ is \emph{primitive} if $u = v^k = \underbrace{v v \cdots v}_{k \text{ times}}$ implies $k=1$.

Let $x \in \R^{+} \setminus \Q$ with its continued fraction expansion equal to $[x_0,x_1,x_2,\ldots]$.
Its \emph{LR representation} is the following infinite word over the alphabet $\{L,R\}$:
\[
\bv = R^{x_0} L^{x_1} R^{x_2} L^{x_3} \ldots \quad \in \{L,R\}^\N.
\]
In what follows, we identify a number $x$ with its LR expansion and simply write $x = \bv$.
For example, we have $1 + \sqrt{2} = [\overline{2}] = \overline{R^2L^2}$.

\begin{remark}
The LR representation is originally connected with the Stern-Brocot tree.
The choice of letters $L$ and $R$ also follows from this connection: the two letters stand for ``Left'' and ``Right'' in the tree.
For more information about the relation between the Stern-Brocot tree and continued fractions see for instance \cite{niqui}.
\end{remark}

Let $V$ be a finite word.
A \emph{run} in $V$ is a contiguous subsequence of maximal length which consists of a single letter.
That is, it is the longest repetition of one letter, sometimes also called a tandem array.
The number $\sigma(V)$ denotes the number of all runs in $V$.
For instance, we have $\sigma(LLRRRRL) = 3$.

\subsection{M\"{o}bius transformation and finite state transducers}

In \cite{Gosper}, the author introduces an algorithm that calculates the continued fraction of $h_M(x)$ using the continued fraction of $x$.
The general idea of the algorithm is the following: read as many partial coefficients of $x$ so that we are able to decide on the first partial coefficient of $h_M(x)$ and output it.
The reading phase is usually called \emph{absorption}, the writing phase \emph{emission}.
Then, if needed, continue absorbing the partial coefficients of $x$ and emit the second partial coefficient of $h_M(x)$ when possible.
Repeat the whole procedure: if there are no coefficients to absorb, emit the rest of the output.
The details and more results on the algorithm were given later by Raney, in \cite{Raney1973}.
The main idea of the algorithm is the same but it uses LR representations instead of continued fractions expansions.
In this article, we use the latter approach and work with LR representations since they allow capturing more details of the algorithm.
In what follows, we sum up the needed results of Raney.
We also refer the reader to a more general concept of this idea in \cite[Chapter 5]{Kurka2016}, and for refinements of the results for continued fractions in \cite{LiSta}.

Since for every positive integer $d$ we have $h_{dM}(x) = h_M(x)$, we shall work only with matrices $M$ such that the greatest common divisor of all its elements is $1$.

Following \cite{Raney1973}, we define some special sets of matrices $2 \times 2$ having a key role in the computation of Möbius transformations.

\begin{definition}
For $n \in \N$, $n \neq 0$ we set
\[
\D{} = \left\{ A \in \N^{2,2} \colon \det(A) = n, \gcd(A) = 1 \right \},
\]
where $\gcd(A)$ denotes the greatest common divisor of all elements of $A$.

Furthermore, we define the three following subsets of $\D{}$:
\begin{align*}
\RB{} & = \left\{ \begin{pmatrix} a & b  \\ c & d \end{pmatrix} \in \D{} \colon a > c \text{ and } d > b  \right \}, \\
\CB{} & = \left\{ \begin{pmatrix} a & b  \\ c & d \end{pmatrix} \in \D{} \colon a > b \text{ and } d > c  \right \}, \\
\DB{}& = \RB{} \cap \CB{}.
\end{align*}
\end{definition}

The names of the three above defined sets are abbreviations for ``row-balanced'', ``column-balanced'' and ``double-balanced'', respectively.

For all $n$, the sets $\RB{}$, $\CB{}$, and $\DB{}$ are finite.
If $n$ is a prime number, then by Corollary 4.7 of \cite{Raney1973}, we have $\# \DB{}= n$.

We study the period of eventually periodic continued fractions and therefore we do not need the prefix of  LR representations of the studied numbers, only the tail is important.
In the view of this, the following theorem tells us that we may consider M\"obius transformations associated to a matrix from $\DB{}$.

\begin{theorem}[\cite{Raney1973,LiSta}] \label{thm:dostanu_se_do_Dn}
Let $x$ be an irrational number, $h_M$ a M\"obius transformation, $\gcd(M) = 1$ and $n = \left| \det M \right |$.
There exists an algorithm to construct a matrix $N \in \DB{}$ and a positive irrational number $y$ such that the continued fraction expansion of $h_N(y)$ and the continued fraction expansion of $h_M(x)$ have the same tail.

In particular, if $x$ is a quadratic irrational number, we have
\[
\per(h_M(x)) = \per(h_N(y)).
\]
\end{theorem}

\emph{A finite state transducer} is the quadruple $(Q,\A,\B,\delta)$ where $Q$ is a finite set of states, $\A$ is the input alphabet, $\B$ is the output alphabet and $\delta \subseteq Q \times \A^* \times\B^* \times Q$ is the transition relation.
The transitions are also called \emph{edges} of this transducer.
The first state in the transition relation is the starting state of this edge.
The word $v \in \A^{*}$ is the input label of this edge.
The word $w \in \B^{*}$ is the output label of this edge and the second state in the relation is the ending state of this edge.

Raney shows that once the problem is transformed to involve a M\"obius transformation of a positive number $y$ with a matrix $N \in \DB{}$, there exists a finite state transducer depending on $n$, denoted $\T{}$, that can be used to determine the LR expansion of $h_N(y)$.
Namely, the input word of this transducer is the LR expansion of $y$, the initial state is given by $N$, and the output word is the LR expansion of $h_N(y)$.

As we are interested only in the repetend of $h_M(x)$, which is the same for $h_N(y)$, we can focus only on the calculation using the transducer $\T{}$.
Thus, we refrain from giving more details on the last theorem and continue with the description of $\T{}$ and its properties.

\subsubsection{Matrices \texorpdfstring{$L$}{L} and \texorpdfstring{$R$}{R} }

We start by identifying the set of all finite words over $\{L,R\}$ with the elements of $\Dj$.
Let $\mu: L \mapsto \begin{pmatrix}
1 & 0 \\ 1 & 1
\end{pmatrix}, R \mapsto \begin{pmatrix}
1 & 1 \\ 0 & 1
\end{pmatrix}$
and for all $V, W \in \{L,R\}^*$, we have $\mu(VW) = \mu(V)\mu(W)$.

\begin{proposition}[\cite{Raney1973}] \label{thm:jednoznacnost_v_D1}
The mapping $\mu$ is an isomorphism of $\{L,R\}^*$ and $\Dj$.
\end{proposition}

Since $\mu$ is an isomorphism, we shall identify the letters $L$ and $R$ with the two matrices, i.e., we shall consider
\[
L = \begin{pmatrix}
1 & 0 \\ 1 & 1
\end{pmatrix}
\quad \text{ and } \quad
R = \begin{pmatrix}
1 & 1 \\ 0 & 1
\end{pmatrix}.
\]

In what follows, we often need to deduce some claims from matrix equations and these equations include mainly the matrices $L$, $R$ and their inverses.
We give the two following lemmas to be used in these cases.

\begin{lemma}
Let $i,j \in \Z$.
We have
\begin{eqnarray}
LR^{-j}L^{-1} &=& RL^{j}R^{-1}, \label{eq:bubli_Li_vpravo} \\
R^{-1}L^{-i}R &=& L^{-1}R^iL, \label{eq:bubli_R_vpravo} \\
L \left( L^{i+1}R^{j+1} \right)^{-1} R &=& RL^jR^iL. \label{eq:bubliani_D1_1}
\end{eqnarray}
\end{lemma}

\begin{proof}
The first two equalities may be verified by direct computation.
The last equality follows from the combination of \eqref{eq:bubli_R_vpravo} and \eqref{eq:bubli_Li_vpravo} as follows:
\[
L \left( L^{i+1}R^{j+1} \right)^{-1} R = L R^{-j-1} L^{-i-1} R \overset{\eqref{eq:bubli_R_vpravo}}{=} L R^{-j}L^{-1}R^{i+1}L \overset{\eqref{eq:bubli_Li_vpravo}}{=} RL^jR^iL. \qedhere
\]
\end{proof}

\begin{lemma} \label{st:bublani_D1}
If $W \in \Dj$,
then
\[
L \left( LWR \right)^{-1} R \in \Dj \quad \text{ and } \quad  R \text{ is a prefix of } L \left( LWR \right)^{-1} R.
\]
\end{lemma}

\begin{proof}
Let
\[
LWR = L^{i_0+1}R^{j_0+1}L^{i_1+1}R^{j_1+1} \cdots L^{i_k+1}R^{j_k+1}
\]
for some $k \geq 0$ and $i_s,j_s \geq 0$ for all $s \in \{0,\dots,k\}$.
We have
\[
L \left( LWR \right)^{-1} R =
\]
\[= L \left( L^{i_k+1}R^{j_k+1} \right)^{-1} R R^{-1} L^{-1} L \left( L^{i_{k-1}+1}R^{j_{k-1}+1} \right)^{-1} R R^{-1} \cdots L^{-1} L \left( L^{i_0+1}R^{j_0+1} \right)^{-1} R.
\]
Using $k$ times \eqref{eq:bubliani_D1_1} on the right side of the last equality, we obtain
\[
L \left( LWR \right)^{-1} R = R L^{j_k} R^{i_k} L R^{-1} L^{-1} R L^{j_{k-1}}R^{i_{k-1}} L R^{-1} \cdots L^{-1} R L^{j_0}R^{j_0} L.
\]
Since by \eqref{eq:bubli_Li_vpravo} we have $L R^{-1} L^{-1} R = RL$, we conclude
\[
L \left( LWR \right)^{-1} R = R L^{j_k} R^{i_k} RL L^{j_{k-1}}R^{i_{k-1}} RL \cdots L^{j_0}R^{j_0} L \in \Dj. \qedhere
\]
\end{proof}

\subsubsection{Transducers \texorpdfstring{$\T{}$}{Tn}}

\begin{theorem}[{\cite[Theorem 5.1]{Raney1973}}] \label{thm:Raney_edge}
Let $M \in \RB{}$.
For all $V \in \{L,R\}^*$ such that
\begin{itemize}
  \item $MV \not \in \RB{}$, and
  \item $MV_1 \in \RB{}$ for every proper prefix $V_1$ of $V$,
\end{itemize}
there exists a unique non-empty word $W \in \{L,R\}^*$ and $N \in \DB{}$ such that
\begin{equation} \label{eq:Raney_edge}
MV = WN.
\end{equation}
\end{theorem}

Theorem 5.1 in \cite{Raney1973} does not say that the word $W$ is unique, however this property follows directly from the equation $W = MVN^{-1}$.

We have used some statements of \cite{Raney1973} about the sets $\D{}$ which are in \cite{Raney1973} defined without the condition that $\gcd(A) = 1$.
The validity of these statements for our definition of $\D{}$ follows from Corollary 8.4 in \cite{Raney1973}.

Based on the last theorem, we may now construct the transducer $\T{}$.

The definition of $\T{}$ is as follows:
\begin{enumerate}
    \item the set of states of $\T{}$ equals $\DB{}$;
    \item the set of transitions between states is given by \Cref{thm:Raney_edge}: there is a transition from $M$ to $N$ if $MV = WN$ for some $V,W \in \Dj$ with $MV \not \in \RB{}$ and $MV_1 \in \RB{}$ for every proper prefix $V_1$ of $V$. The input word of the transition is $V$, the output word is $W$.
\end{enumerate}

To ease our notation, the transition from $M$ to $N$ with input $V$ and output $W$ is denoted by
\[
\Tedge{M}{V}{W}{N}.
\]

Let $M$ and $N$ be two states of $\T{}$ such that there is a sequence of transitions starting at $M$ and ending at $N$ in $\T{}$ with concatenation of respective input words $V$ and output words $W$. We write
\[
\Twalk{M}{V}{W}{N}.
\]
We call this sequence a \emph{walk}.
Concatenating the matrix relations of all transition in the walk we obtain the relation $MV = WN$.
To ease the notation we also allow $V$ to be the empty word, which implies that $W$ is also empty and $M = N$.
If we do not need to know the concrete input or output word, we write $\bullet$ on its position.

Given a walk $\Twalk{M}{V}{W}{N}$, we shall write for instance
\[
\Twalk{M}{V}{W}{N} = \Twalk{M}{V_1}{W_1}{\Tedge{M_1}{V_2}{W_2}{\Tedge{M_2}{V_3}{W_3}{N}}}
\]
to specify some decomposition of the walk. If a walk repeats, we shall also write for instance
\[
\Twalk{M}{V}{W}{\Twalk{M}{V}{W}{M}} = \left( \Twalk{M}{V}{W}{M} \right)^2.
\]

In what follows, let $A_n = \mat{n}{0}{0}{1}, \assoc{A_n} = \mat{1}{0}{0}{n}$.
We have $A_n, \assoc{A_n} \in \DB{}$ for all $n$ and for $n = 2$, we have $\mathcal{DB}_2 = \left\{ A, \assoc{A} \right\}$.
See \Cref{fig:RaneyT_2} which depicts $\T[2]{}$ and \Cref{fig:table_T_3} showing the transition labels of $\T[3]$.

\begin{figure}[!htb]
  \centering
\begin{tikzpicture}[auto, initial text=, >=latex]
\node[state, accepting] (v0) at (3.000000, 0.000000) {$A_2$};
\node[state, accepting] (v1) at (-3.000000, 0.000000) {$\assoc{A_2}$};
\path[->] (v0) edge[loop above] node {$R|R^2$} ();
\path[->] (v0) edge[loop below] node {$L^2|L$} ();
\path[->] (v0.185.00) edge node[rotate=360.00, anchor=north] {$LR|RL$} (v1.355.00);
\path[->] (v1.5.00) edge node[rotate=0.00, anchor=south] {$RL|LR$} (v0.175.00);
\path[->] (v1) edge[loop below] node {$L|L^2$} ();
\path[->] (v1) edge[loop above] node {$R^2|R$} ();
\end{tikzpicture}
  \caption{Transducer $\T[2]{}$.}
  \label{fig:RaneyT_2}
\end{figure}

\begin{table}[!htb]
  \centering
$\begin{matrix} & A_3 = \begin{pmatrix} 3 & 0 \\
0 & 1 \\
\end{pmatrix}  & B = \begin{pmatrix} 2 & 1 \\
1 & 2 \\
\end{pmatrix}  & \assoc{A_3} = \begin{pmatrix} 1 & 0 \\
0 & 3 \\
\end{pmatrix}  & \\
\begin{pmatrix}3&0\\
0&1\\
\end{pmatrix} & R|R^{3} ,  L^{3}|L & LR|R & L^{2}R|RL^{2}\\
\begin{pmatrix}2&1\\
1&2\\
\end{pmatrix} & L|LR &  & R|RL\\
\begin{pmatrix}1&0\\
0&3\\
\end{pmatrix} & R^{2}L|LR^{2} & RL|L & L|L^{3} , R^{3}|R\\
\end{matrix}$
  \caption{Transitions in the transducer $\T[3]{}$.}
  \label{fig:table_T_3}
\end{table}

\subsubsection{Symmetries of the transducer \texorpdfstring{$\T{}$}{Tn}}

The transducer $\T{}$ possesses some symmetries that we shall use later.

Let $M = \begin{pmatrix}
a & b \\ c & d
\end{pmatrix}$.
The matrix $\assoc{M}$ \emph{associated to} $M$ is given by
\[
\assoc{M} = \begin{pmatrix}
d & c \\ b & a
\end{pmatrix} =
\begin{pmatrix}
0 & 1 \\
1 & 0
\end{pmatrix}
M
\begin{pmatrix}
0 & 1 \\
1 & 0
\end{pmatrix}
.
\]
Clearly, for all matrices $M$ and $N$, we have $\assoc{MN} = \assoc{M}\assoc{N}$.

Note that since we identified the letters $L$ and $R$ with matrices and due to \Cref{thm:jednoznacnost_v_D1}, this operation is also defined for any word over $\{L,R\}$.
The following simple identities are given in \cite[Theorem 7.1]{Raney1973} for $M \in \Dj$, i.e., words over $\{L,R\}$:
\begin{align*}
M & = L^{i_0}R^{i_1}L^{i_2} \cdots L^{i_\ell}, \\
\assoc{M} & = R^{i_0}L^{i_1}R^{i_2} \cdots R^{i_\ell}, \\
(\assoc{M})^T = \assoc{M^T} & = L^{i_\ell}R^{i_{\ell-1}} \cdots R^{i_1}L^{i_0}, \\
M^T & = R^{i_\ell}L^{i_{\ell-1}} \cdots L^{i_1}R^{i_0},
\end{align*}
where $i_0,i_1,\ldots,i_\ell$ are nonnegative integers.
These properties imply a symmetry of $\T{}$ in the sense of the following claim.

\begin{proposition} \label{prop:assoc_sym}
If the transition $\Tedge{M}{V}{W}{N}$ exists, then the transitions $\Tedge{\assoc{M}}{\assoc{V}}{\assoc{W}}{\assoc{N}}$ and $\Tedge{N^T}{W^T}{V^T}{N^T}$ exist.
\end{proposition}

\subsubsection{Relation between repetends and runs}

The following lemma, exhibiting the relation between the number of runs in the repetend of the LR representation of $x$ and the period of its continued fraction, is a direct corollary of the definition of LR representation.

\begin{lemma} \label{le:vypocet_per}
Let $V$ be such a repetend of $x$ whose first and last letter are different.
We have
\[
 \per(x) = \begin{cases}
 \frac{\sigma(V)}{2} & \text{ if $V = V_1\assoc{V_1}$ for some $V_1 \in \Dj$,} \\
 \sigma(V) & \text{ otherwise.}
 \end{cases}
\]
\end{lemma}

To work with the matrix representation of words, we shall need the connection between the number of runs in some word over $\{L,R\}$ and its matrix representation.

\begin{lemma} \label{le:pocet_LRzmen}
Let $W = LW'R, W = \mat{a}{b}{c}{d}$ and $W,W' \in \Dj$. We have
\[
\sigma(W) = 2 \left\lfloor \frac{\xi(a,c)}{2} \right\rfloor + 2.
\]
\end{lemma}

\begin{proof}
Let $f,e$ be positive integers such that $W = W_1L^eR^f$ for some $W \in \Dj$.

We shall proceed by induction on $\sigma(W)$.

For $\sigma(W) = 2$, we have $W = L^eR^f = \mat{1}{f}{e}{ef+1}$ and since $\xi(1,e)=1$ the claim holds.

Assume now the claim holds for $\sigma(W) = k$.
Note that for $W = \mat{a}{b}{c}{d}$, as $W$ starts with $L$ and ends with $R$, we have $c > a \geq 1$ or $c=a=1$.
Let $g,h$ be positive integers.
We have $\sigma(L^gR^hW) = k+2$ and  we have
\[
L^gR^hW = \mat{c h + a}{d h + b}{ g(ch + a)  + c}{g (dh + b)  + d}.
\]
If $c > a$, then $\xi(c h + a, g(ch + a)  + c) = \xi(a,c)+2$ and the claim follows.
If $c=a=1$, then in fact $W = LR^f$, $\xi(h + 1, g(h + 1)  + 1) = 2$ and $\sigma(L^gR^hLR^f) = 4 = 2 \left\lfloor \frac{\xi(h + 1, g(h + 1)  + 1)}{2} \right\rfloor + 2$.
\end{proof}

\begin{example}
The word $W = L^2RLR^{3}$ is of the form from the above lemma and we have $W = \mat{2}{7}{5}{18}$.
Therefore, $ 2 \left\lfloor \frac{\xi(a,c)}{2} \right\rfloor + 2 =2 \left\lfloor \frac{\xi(2,5)}{2} \right\rfloor + 2 = 2\left\lfloor \frac{2}{2} \right\rfloor + 2 = 4 = \sigma(W)$.
\end{example}

\subsection{Properties of the transducer \texorpdfstring{$\T{}$}{Tn}}

In what follows, we suppose that $n$ is a fixed integer.

\subsubsection{Input and output words of the edges in the transducer \texorpdfstring{$\T{}$}{Tn}}

First, we investigate the input words on outgoing edges from some state of the transducer $\T{}$.

\begin{lemma} \label{le:vstupni_slova}
If $X \in \RB{}$,
then $XL \not \in \RB{}$ or $XR \not \in \RB{}$.
\end{lemma}

\begin{proof}
Let $X = \mat{a}{b}{c}{d} \in \RB{}$.
Therefore, $a>c$, $d>b$ and $XL =\mat{a+b}{b}{c+d}{d}, XR = \mat{a}{a+b}{c}{c+d}$.
We have $a+b \geq c+d$ and then $XR \not \in \RB{}$, or $a+b \leq c+d$ and then $XL \not \in \RB{}$.
\end{proof}

A simple consequence of \Cref{thm:Raney_edge} is that the set of all input word of an outgoing edge of a state of $\T{}$ is a prefix code (no element is a prefix of another).
The last lemma implies that the lengths of these input words are $1,2,\ldots,\ell-1,\ell,\ell$ for some $\ell$.

\begin{example}
All transitions in the transducer $\mathcal{T}_{14}$ are given in \Cref{tab:T_14}.
\end{example}

\begin{sidewaystable}
  \centering
  \includegraphics[width=\textwidth]{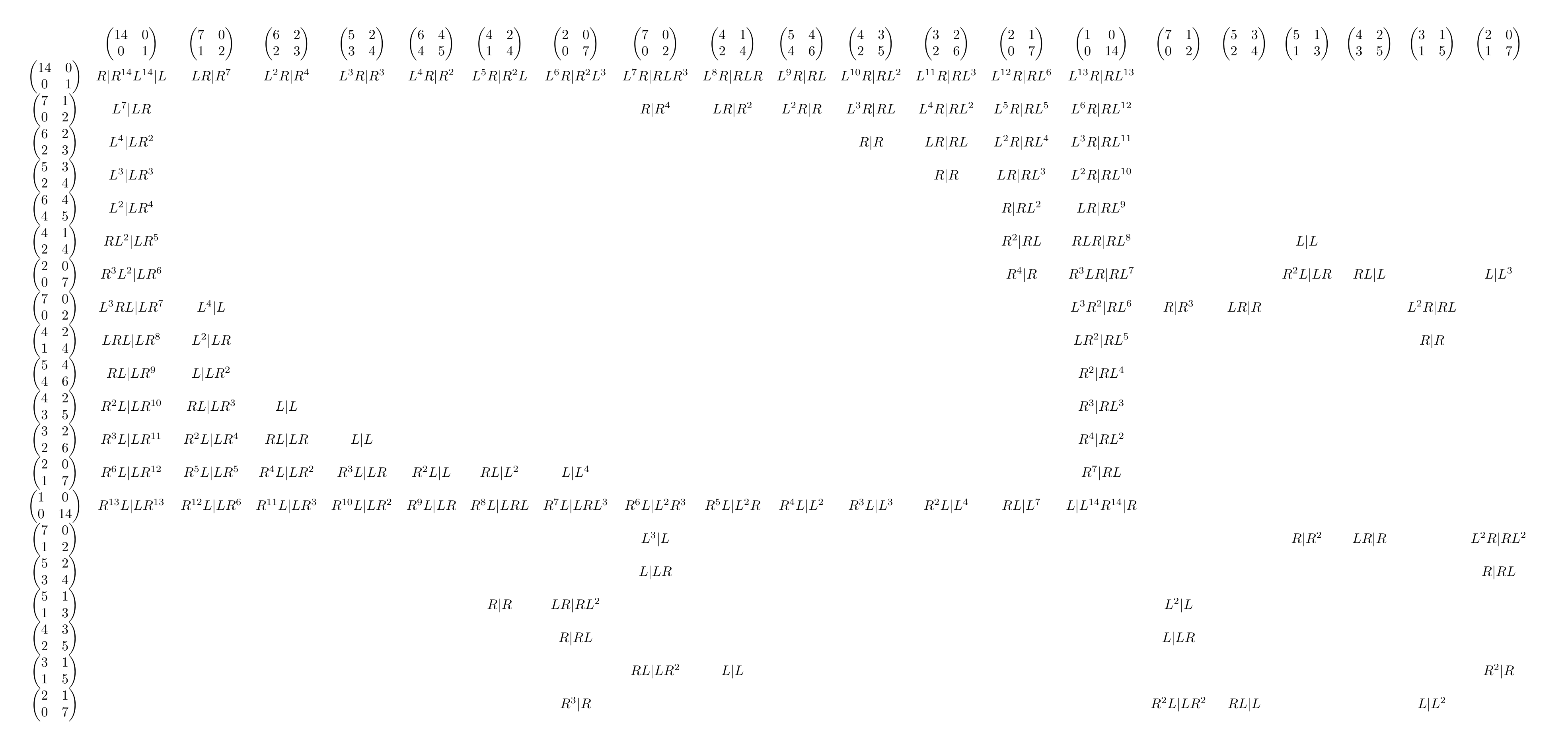}
  \caption[angle=90]{Table of transitions in the transducer $\mathcal{T}_{14}$.}
\label{tab:T_14}
\end{sidewaystable}

The following lemma shows that there is no edge with input label having suffix $R^{n+1}$ or $L^{n+1}$

\begin{lemma} \label{le:max_n_jednoho_pismene}
Let $Q \in \{L, R\}$.
Let $X \in \RB{}$ and let $i$ be the minimal positive integer such that $XQ^i \not \in \RB{}$.
We have $i \leq n$.
\end{lemma}

\begin{proof}
Let $Q = L$ and let $X = \begin{pmatrix}
a & b \\ c & d
\end{pmatrix} \in \RB{}$.
We have
\[
XL^i = \begin{pmatrix}
a + bi & b \\ c + di & d
\end{pmatrix}.
\]
Assume $i$ is minimal such that $XL^i \not \in \RB{}$. Therefore, $XL^{i-1} \in \RB{}$.
Since $d > b$, we have $a + b(i-1) > c + d(i-1)$, which implies
\[
i -1 < \frac{a - c}{d - b}.
\]
Using $a = \frac{n+bc}{d}$, we obtain
\[
\frac{a - c}{d - b} = \frac{n}{d(d-b)} - \frac{c}{d} \leq n.
\]

The proof for $Q = R$ is analogous.
\end{proof}

The next lemma shows a simple link of input and output words.

\begin{lemma}[{\cite[Theorem 5.1]{Raney1973}}] \label{le:stejna_pismena}
Let $\Tedge{M}{VQ_1}{Q_2W}{N}$ with $Q_1,Q_2 \in \{L,R\}$.
We have $Q_1 = Q_2$.
\end{lemma}

\subsubsection{Sets \texorpdfstring{$\LS{}$}{LS} and \texorpdfstring{$\RS{}$}{RS}}
\label{sec:LS_n}

In what follows, we show that an important role is played by input words with long runs of the same letter.
We now introduce two sets of matrices that are always visited when reading such input words.

\begin{definition} \label{def:LS_a_RS}
Let $\LS{}$ respectively $\RS{}$ denote the subset of $\DB{}$ such that $X \in \LS{}$ respectively $X \in \RS{}$ if there exists a walk $\Twalk{X}{L^i}{L^j}{X}$ respectively $\Twalk{X}{R^i}{R^j}{X}$ for some $i,j > 0$.
\end{definition}

The names of the sets $\LS{}$ and $\RS{}$ are abbreviations for ``$L$-special'' and ``$R$-special''.
For example, we have $A_n \in \LS{}$ since $\Tedge{A_n}{L^n}{L}{A_n}$ and $A_n \in \RS{}$ since $\Tedge{A_n}{R}{R^n}{A_n}$.

The above definition says that the matrices in set $\LS{}$ ($\RS{}$) can be visited several times while reading the same run of letters.
Moreover, the following lemma shows that these matrices are the only ones with such a property.

\begin{lemma} \label{le:char_toceni_na_L}
Let $M = \mat{a}{b}{c}{d} \in \DB{}$ and $i > 0$.
We have
\[
\Twalk{M}{L^i}{W}{M}
\]
for some $W \in \Dj$ if and only if $M \in \LS{}$ (i.e., $W = L^j$).

Moreover, the minimal value of such integer $i$ equals $\min \left\{ k > 0 \colon \frac{kd}{a} \in \N \right\}$ and is less than or equal to $n$.
\end{lemma}

\begin{proof}
The walk $\Twalk{M}{L^i}{W}{M}$ exists if and only if $W  \in \Dj$.
We have
\[
W = ML^iM^{-1} = \mat{-\frac{b d i - b c + a d}{b c - a d}}{\frac{b^{2} i}{b c - a d}}{
-\frac{d^{2} i}{b c - a d}}{\frac{b d i + b c - a d}{b c - a d}}.
\]
As $\frac{b^{2} i}{b c - a d} = \frac{b^2i}{-n} \leq 0$ we have $\frac{b^2i}{bc-ad} \in \N \iff b = 0$.
Thus
\[
W = \mat{1}{0}{\frac{id}{a}}{1},
\]
and we conclude that $W \in \Dj$ if and only if $b = 0$ and $\frac{id}{a} \in \N$.
In other words, $W  = L^{\frac{id}{a}}$, which is by definition if and only if $M \in \LS{}$.

Let $i$ has the minimal possible value, i.e., $i = \min \left\{ k > 0 \colon \frac{kd}{a} \in \N \right\}$.
Since $n = ad$, we have $i \leq a \leq n$.
\end{proof}

\begin{example}
In the transducer $\mathcal{T}_{14}$, we have $\Twalk{M}{L^7}{W}{M}$ where $M = \mat{7}{0}{0}{2}$ and $W = L^2$.
Since $\Twalk{M}{L^7}{L^2}{M} = \Tedge{M}{L^4}{L}{\Tedge{N}{L^3}{L}{M}}$ where $N = \mat{7}{0}{1}{2}$, the minimal $i$ such that $\Twalk{M}{L^i}{W}{M}$ is $i = 7$.
This is also the minimal positive integer such that $\frac{id}{a} = \frac{2i}{7} \in \N$.
\end{example}

The following characteristic property of the matrices in the set $\LS{}$ follows from the proof of the last lemma.

\begin{corollary} \label{cor:tvar_LS}
Let $M = \mat{a}{b}{c}{d} \in \DB{}$.
We have $M  \in \LS{}$ if and only if $b = 0$.
\end{corollary}

\begin{example}
We have already seen that the matrices $\mat{7}{0}{0}{2}, \mat{7}{0}{1}{2} \in \mathcal{LS}_{14}$ and that $A_n \in \LS{}$, which corresponds with the fact that these matrices have the element on position 1,2 equal to 0. Moreover, by the symmetric version of this corollary, we have $M \in \RS{}$ if and only if $c = 0$, which corresponds with $A_n \in \RS{}$.
According to this condition, we can also see that $\mat{7}{0}{0}{2} \in \RS{}$. Indeed, we have $\Twalk{\mat{7}{0}{0}{2}}{R^2}{R^7}{\mat{7}{0}{0}{2}}$.
\end{example}

The last lemma and corollary imply the following statement.

\begin{corollary} \label{le:LS_nacteni_n}
For $M \in \LS{}$ there exists $s > 0$ such that $\Twalk{M}{L^n}{L^s}{M}$.
\end{corollary}

\begin{example}
As already mentioned, we have $\Twalk{M}{L^7}{L^2}{M} $ where $M = \mat{7}{0}{0}{2} \in \mathcal{LS}_{14}$.
The last corollary shows that we have also $\Twalk{M}{L^{14}}{L^4}{M} = (\Twalk{M}{L^7}{L^2}{M})^2$.

As stated above, we have  $\Twalk{M}{L^7}{L^2}{M} = \Tedge{M}{L^4}{L}{\Tedge{N}{L^3}{L}{M}}$ where $N = \mat{7}{0}{1}{2}$.
This means that there is a connection between the matrices $\mat{7}{0}{0}{2}$ and $\mat{7}{0}{1}{2} \in \mathcal{LS}_{14}$.
The classes of matrices in $\LS{}$ with this connection are described in the following lemma.
\end{example}

\begin{lemma} \label{le:trida_L}
Let $M, N \in \LS{}$, $M = \mat{a}{0}{c}{d}$ and $N = \mat{a'}{0}{c'}{d'}$.
We have
\[
\Twalk{M}{L^i}{W}{N} \text{ for some $i > 0$ and $W \in \Dj$} \quad \Longleftrightarrow \quad a' = a, d' = d, c' \equiv c \pmod{ \gcd(a,d)}.
\]
\end{lemma}

\begin{proof}
Let $\Twalk{M}{L^i}{W}{N}$ for some $i > 0$ and $W \in \Dj$.
\Cref{le:char_toceni_na_L} implies that we may take $W = L^j$ for some $j > 0$ and $i \leq n$.
We have
\[
N = L^{-j} M L^i = \mat{a}{0}{c + id - ja}{d} \in \LS{} \subseteq \DB.
\]
Since we have $id - ja = k \gcd(a,d)$ for some $k \in \Z$, the first implication is proven.

Assume now $a' = a, d' = d, c' = c + k\gcd(a,d)$.
The case $c = c'$ is trivial.
For $c \neq c'$ we find $i',j'$ with $i'j' < 0$ such that $k\gcd(a,d) = i'd + j'a $.
It implies $ML^{i'} = L^{-j'}N$.
If $j' < 0$, we are finished.
If $j' > 0$, we multiply by $L^{rn}$ from the right to obtain
\[
ML^{i'+rn} = L^{-j'}NL^{rn} = L^{-j'+rs}N,
\]
where $s$ is the positive integer such that $\Twalk{N}{L^n}{L^s}{N}$.
A choice of $r$ such that $-j'+rs > 0$ implies $\Twalk{M}{L^{i'+rn}}{L^{-j'+rs}}{N}$.
\end{proof}

The existence of a walk $\Twalk{M}{L^i}{W}{N}$ is in fact an equivalence relation between $M$ and $N$.
For each class of this equivalence, we pick a suitable representative in the following definition.

\begin{definition}\label{def:LE_n}
Let $\LE{}$ $(\RE{})$ denote the subset of $\LS{}$ $(\RS{})$ such that $M = \mat{a}{0}{c}{d} \in \LE{}$ $(N = \mat{a}{b}{0}{d} \in \RE{})$ if and only if $c < \gcd(a,d)$ $(b< \gcd(a,d))$.
\end{definition}

The names $\LE{}$ and $\RE{}$ are abbreviations for ``$L$-exceptional'' and ``$R$-exceptional''.
For instance, $\mat{7}{0}{0}{2} \in \LE[14]{}$ and $\mat{7}{0}{1}{2} \not \in \LE[14]{}$ because $\gcd(a,d) = 1$.

Combining this definition with the last lemma, we immediately obtain the following corollary which says that the suitable representative is unique.

\begin{corollary} \label{cor:jednoznacnost_LE}
Let $M \in \LS{}$.
There is exactly one $N \in \LE{}$ such that $\Twalk{M}{L^i}{W}{N}$ exists for some $i \geq 0$ and $W \in \Dj$.
\end{corollary}

\begin{proof}
Let $M = \mat{a}{0}{c}{d}$.
\Cref{le:trida_L} implies that the walk $\Twalk{M}{L^i}{W}{N'} $ exists if and only if $N' = \mat{a}{0}{c'}{d}$ and $c' \equiv c \pmod{ \gcd(a,d)}$.
As there is exactly one such $c' < \gcd(a,d)$, the claim follows from \Cref{def:LE_n}.
\end{proof}


For each state $M \in \LE{}$, we shall need to know the least number $i$ such that we can get from $M$ to $M$ by reading $L^i$ as an input word.

\begin{definition} \label{def:nu}
For $M \in \LE{}$ we set
\[
\LclassMax{M} = \min \left\{ i > 0 \colon \Twalk{M}{L^i}{L^j}{M} \text{ exists for some } j \right\}.
\]
For $M \in \RE{}$, we define $\RclassMax{M}$ analogously:
\[
\RclassMax{M} = \min \left\{ i > 0 \colon \Twalk{M}{R^i}{R^j}{M} \text{ exists for some } j \right\}.
\]
\end{definition}

\begin{example}
For the state $M = \mat{7}{0}{0}{2} \in \LE[14]{}$, we have $\LclassMax{M} = 7$.
Moreover, we have $M \in \RE[14]{}$ and $\RclassMax{M} = 2$.
\end{example}

The purpose of the definition of the sets $\LE$ and $\RE$ is in the following lemma.

\begin{lemma} \label{le:spadnu_do_LS}
Let $Q \in \DB{}$.
There exists exactly one matrix $Z \in \LE{}$ such that $\Twalk{Q}{L^k}{W}{Z}$ for some $k \geq 0$ and $ W \in \Dj$.

Moreover, the integer $k$ can be chosen such that $k \leq n$.
\end{lemma}

\begin{proof}
\Cref{le:char_toceni_na_L} and $\# \DB{}< + \infty$ imply  that there exist $j \geq 0$, $M \in \LS{}$ and $W_1 \in \Dj$ such that $\Twalk{Q}{L^j}{W_1}{M}$.
Let $j$ be the least possible.
Such $M$ is unique (depending on $Q$ only).
\Cref{cor:jednoznacnost_LE} implies that there is exactly one $Z \in \LE{}$ such that $\Twalk{M}{L^{j'}}{W_2}{Z}$ for some $j' \in \N$ and $W_2 \in \Dj$.
It follows that $\Twalk{Q}{L^{j+j'}}{W_1W_2}{Z}$.
The uniqueness of $Z$ follows from the uniqueness of $M$ and the uniqueness of $Z$ by \Cref{cor:jednoznacnost_LE}.

To prove the second part, we suppose that the walk $\Twalk{Q}{L^{k'}}{W'}{Z}$ exists for some $k' \in \N$ and $W' \in \Dj$.
Let $Q = \mat{a}{b}{c}{d}$ and $Z = \mat{a'}{0}{c'}{d'}$. It implies that
\[
W' = QL^{k'}Z^{-1} = \mat{\frac{b {{d'}} k' - b {{c'}} + a {{d'}}}{{{a'}} {{d'}}}}{\frac{b}{{{d'}}}}{\frac{d {{d'}} k' - {{c'}} d + c {{d'}}}{{{a'}} {{d'}}}}{\frac{d}{{{d'}}}} \in \Dj.
\]
Set $k$ such that $k \equiv k' \pmod n$ and $k \in \{1,\ldots,n\}$.
Since $\frac{b {{d'}} k' - b {{c'}} + a {{d'}}}{{{a'}} {{d'}}}$ and $\frac{d {{d'}} k' - {{c'}} d + c {{d'}}}{{{a'}} {{d'}}}$ are both integers, $a'd' = n$, and $d' > c'$, we have that $\frac{b {{d'}} k - b {{c'}} + a {{d'}}}{{{a'}} {{d'}}}$ and $\frac{d {{d'}} k - {{c'}} d + c {{d'}}}{{{a'}} {{d'}}}$ are positive integers.
Therefore, $W = QL^{k}Z^{-1}$ is a matrix of nonnegative integer elements.
We verify by direct calculation that $\det (W) = 1$, and thus $W \in \Dj$ and the walk $\Twalk{Q}{L^k}{W}{Z}$ exists.
\end{proof}

The last lemma says that if we start in an arbitrary state of the transducer $\T{}$ and we read a long enough run of $L$'s (at most of length $n$) of the input word, we end in some state from $\LE{}$.
Moreover, if we continue to read only the letters $L$, we can attain only the states from $\LS{}$ as follows from \Cref{le:char_toceni_na_L}, and, in particular, we have to return to the same state from $\LE{}$ (\Cref{cor:jednoznacnost_LE}).

\begin{example}
If we take $Q = \mat{4}{2}{1}{4} \in \mathcal{DB}_{14}$, we have $\Twalk{Q}{L^5}{LRL}{M}$ where $M = \mat{7}{0}{0}{2} \in \LE{}$ and if we continue to read only $L$'s, we go through the walk  $\Twalk{M}{L^7}{L^2}{M} = \Tedge{M}{L^4}{L}{ \Tedge{N}{L^3}{L}{M}}$ where $N = \mat{7}{0}{1}{2} \in \LS{}$. Similarly, we have $\Twalk{Q}{R^{10}}{R^3L}{\assoc{A}_{14}}$ where $\assoc{A}_{14} = \mat{1}{0}{0}{14} \in \RE{}$ and if we continue to read only $R$'s, we go through the edge $\Tedge{\assoc{A}_{14}}{R^{14}}{R}{\assoc{A}_{14}}$.
\end{example}

\section{Construction of the bound $S_n$} \label{sec:constr}

In the previous section, we have defined the transducer $\T{}$ and stated its important properties that are used in the construction of the upper bound $S_n$ of \Cref{thm:main}.
This section is dedicated to the construction of this bound.

\subsection{Maximalisation of the prolongation}
\label{sec:kappa}

We define the mapping $\kappa: \Dj \to \Dj$ which shall be used to modify the runs of a word such that their length is within a suitable interval.
For $W = Z_0^{i_0}Z_1^{i_1} \cdots Z_k^{i_k} \in \Dj$ with $i_\ell$ nonzero, $Z_\ell \in \left\{ L,R \right\}$, and $Z_{\ell} \neq Z_{\ell+1}$
we set
\[
\kappa(W) = Z_0^{i_0'}Z_1^{i_1'} \cdots Z_k^{i_k'}
\]
where $i_\ell' \equiv i_\ell \pmod{n}$ and $i_\ell' \in \{4n,4n+1,\ldots,5n-1\}$ for all $\ell \in \{0,\dots,k\}$.
For instance, if $n = 10$, we have
\[
\kappa(LR^{10}L^{55}) = L^{41}R^{40}L^{45}.
\]

First, we show why we do not need runs longer than $5n-1$.

\begin{lemma} \label{le:vyfouknuti_n}
Let $\Twalk{M}{V_1L^mV_2}{W}{N}$.
If $m \geq 4n$,
then $\Twalk{M}{V_1L^{m-n}V_2}{W'}{N}$, $\sigma(W') = \sigma(W)$, and $W'$ starts with the same letter as $W$.
\end{lemma}

\begin{proof}
\Cref{le:max_n_jednoho_pismene} implies that there exists an integer $t$ with $0 < t \leq n$ and $Q \in \DB{}$ such that
\[
\Twalk{M}{V_1L^{t}}{W_1}{Q}
\]
for some $W_1 \in \Dj$.

By \Cref{le:spadnu_do_LS}, there exists an integer $k$ with $k \leq n$ and $Z \in \LE{}$ such that
\[
\Twalk{Q}{L^k}{W_2}{Z}
\]
for some $W_2 \in \Dj$.

As $Z \in \LE{} \subseteq \LS{}$ we have
\[
\Twalk{Z}{L^n}{L^s}{Z}
\]
for some positive integer $s$.

Let $q$ and $p$ be the integers such that $p = (m-t - k) \bmod{n}, p < n$ and $m-t-k = qn + p$.
Since $m \geq 4n \geq 2n + k + t$, we have $q \geq 2$ and there exists the walk
\[
\Twalk{Z}{L^pV_2}{W_3}{N}
\]
 and $W = W_1W_2L^{qs}W_3$.

As $q \geq 2$ implies that the walk $\Twalk{Z}{L^n}{L^s}{Z}$ is used at least twice while reading $V_1L^mV_2$,
we conclude that while reading $V_1L^{m-n}V_2$, we take this walk one less time, but at least once.
Thus, we have
\[
\Twalk{M}{V_1L^{m-n}V_2}{W_1W_2L^{(q-1)s}W_3}{N}
\]
and $\sigma(W) = \sigma(W_1W_2L^{qs}W_3) = \sigma(W_1W_2L^{(q-1)s}W_3)$.
\end{proof}

Similarly, we show that we shall not need the input words with runs shorter than $4n$.
We start with a lemma.

\begin{lemma} \label{le:nafouknuti_jedno}
Let $\Tedge{M}{VL^iR^j}{W}{N}$ with $i \geq 0$, $j \geq 1$, $V \in \Dj$, and $M,N \in \DB{}$.
There exist $Q \in \DB{}$, $W_1, W' \in \Dj$ and an integer $t$ with $0 < t \leq n$ such that for all $r \geq 3$
\[
\Tedge{M}{VL^{i+t}}{W_1}{Q} \quad \text{ and } \quad \Twalk{Q}{L^{rn-t}R^j}{W'}{N}
\]
with $\sigma(W') > \sigma(W)$. Moreover, the word $W'$ has suffix $R^{-1}W$.
\end{lemma}

\begin{proof}
The first part of the proof is very similar to the previous proof.
By \Cref{le:max_n_jednoho_pismene}, there exists an integer $t$ with $0 < t \leq n$ and $Q \in \DB{}$ such that
\begin{equation} \label{pr:nafu_w1}
\Tedge{M}{VL^{i+t}}{W_1}{Q}
\end{equation}
for some $W_1 \in \Dj$.

By \Cref{le:spadnu_do_LS}, there exists an integer $k$ with $k \leq n$ and $Z \in \LE{}$ such that
\begin{equation} \label{pr:nafu_w2}
\Twalk{Q}{L^k}{W_2}{Z}
\end{equation}
for some $W_2 \in \Dj$.

Let $r$ be an integer such that $rn \geq t+k$.
Let $W_3(r)$ be the matrix given by
\begin{equation} \label{pr:nafu_W3}
ZL^{rn-t-k}R^j = W_3(r)N.
\end{equation}
We now show that $W_3(r) \in \Dj$.
Notice that this is equivalent to the existence of the following walk: $\Twalk{Z}{L^{rn-t-k}R^j}{W_3(r)}{N}$.

To show $W_3(r) \in \Dj$, we first find $N^{-1}$ using $\Tedge{M}{VL^iR^j}{W}{N}$ as follows
\begin{equation} \label{pr:nafu_Ni}
N^{-1} = R^{-j}L^{-i}V^{-1}M^{-1}W.
\end{equation}

By \eqref{pr:nafu_w1} and \eqref{pr:nafu_w2}, we have $MVL^{i+t+k} = W_1W_2Z$, therefore
\begin{equation} \label{pr:nafu_W12i}
(W_1W_2)^{-1} = Z L^{-i-t-k} V^{-1} M^{-1}.
\end{equation}

Since $Z \in \LE{} \subseteq \LS{}$, we have by \Cref{le:LS_nacteni_n} that $\Twalk{Z}{L^n}{L^s}{Z}$ for some positive integer $s$.
We combine this fact with \eqref{pr:nafu_W3} and obtain
\[
W_3(r) = ZL^{rn-t-k}R^j N^{-1} = L^{rs} Z L^{-t-k} R^j N^{-1}.
\]
We continue by using \eqref{pr:nafu_Ni} and then \eqref{pr:nafu_W12i} to obtain
\begin{align} \label{eq:W_3}
W_3(r) = L^{rs} Z L^{-t-k} R^j N^{-1} & = L^{rs} Z L^{-t-k} R^j R^{-j} L^{-i}V^{-1}M^{-1}W \\ \notag
& = L^{rs} Z L^{-i-t-k}V^{-1}M^{-1}W \\ \notag
& = L^{rs} (W_1W_2)^{-1} W.
\end{align}

We prove now the following claim.
\begin{equation} \label{pr:aux_claim}
\text{ If $W_3(r) \in \Dj$ for some $r \geq 3$, then $W_3(3) \in \Dj$ and $L^s$ is a prefix of $W_3(3)$. }
\end{equation}
Indeed, since for $r \geq 3$ we have $rn > 2n \geq t+k$, $W_3(r) \in \Dj$ implies $\Twalk{Z}{L^{rn-t-k}R^j}{W_3(r)}{N}$.
Let $rn-t-k = q(r)n + p$ with $0 \leq p < n$, i.e., $p = (-t-k) \bmod n$.
Since $\Twalk{Z}{L^n}{L^s}{Z}$, it implies that the walk $\Twalk{Z}{L^{p}R^j}{L^{-q(r)s}W_3(r)}{N}$ exists and $L^{-q(r)s}W_3(r) \in \Dj$.
As $2n \geq t+k$, using \eqref{eq:W_3}, we find $W_3(2) = L^{q(2)s - q(r)s}W_3(r)$.
Since $0 \leq q(2) < q(r)$, we have $W_3(2) = L^{q(2)s - q(r)s}W_3(r) \in \Dj$.
Moreover, we have $q(3)-q(2) = 1$.
Therefore $W_3(3) = L^s W_3(2) \in \Dj$ and $L^s$ is a prefix of $W_3(3)$, and thus \eqref{pr:aux_claim} holds.

We have the two following cases:

\begin{enumerate}
\item  Assume that $W_1W_2$ does not contain $R$.
We may choose the integer $r$ large enough so that $L^{rs} (W_1W_2)^{-1} \in \Dj$, and thus $W_3(r) \in \Dj$ and $W$ is suffix of $W_3(r)$. Because $W$ has by \Cref{le:stejna_pismena} prefix $R$, then also $R^{-1}W$ is a suffix of $W_3(r)$ and therefore also of $W'$.
By \eqref{pr:aux_claim}, we have $W_3(3) \in \Dj$ and $L^s$ is its prefix.
As $R$ is a prefix of $W$, we conclude $\sigma(W_3(3)) = \sigma(W)+1$.
Since $W' = W_2W_3(3)$, we conclude $\sigma(W') > \sigma(W)$.

\item Assume that $W_1W_2$ contains $R$, i.e., $RL^h$ is its suffix for some $h \geq 0$.
Since by \Cref{le:stejna_pismena} the word $W_1$ starts with $L$, we may write $W_1W_2 = LW_4RL^h$ for some $W_4 \in \Dj$.
Using \Cref{st:bublani_D1}, we conclude that $L \left( LW_4R \right)^{-1} R \in \Dj$.
We choose $r$ such that $rs-1 \geq h$.
As $W$ starts with $R$ we have $R^{-1}W \in \Dj$.
We conclude that
\[
W_3(r) = L^{rs} (W_1W_2)^{-1} W = L^{rs-1} L \left( LW_4RL^h \right)^{-1} R R^{-1}W = L^{rs-1-h} L \left( LW_4R \right)^{-1} R R^{-1}W \in \Dj
\]
and $R^{-1}W$ is a suffix of $W_3(r)$ and therefore also of $W'$.

Therefore, by \eqref{pr:aux_claim}, $W_3(3) \in \Dj$ and it starts with $L^s$.
By \Cref{st:bublani_D1}, the word $L \left( LW_4R \right)^{-1} R$ starts with $R$.
It implies that $\sigma(L^{rs} (W_1W_2)^{-1} R) \geq 2$ and moreover, if $L^{rs} (W_1W_2)^{-1} R$ ends with $L$, then $\sigma(L^{rs} (W_1W_2)^{-1} R)\geq 3$.
Together with the facts that if $R^{-1}W$ starts with $R$, then $\sigma(R^{-1}W) = \sigma(W)$ and $\sigma(R^{-1}W) = \sigma(W) - 1$ otherwise, we conclude $\sigma(W_3(3)) > \sigma(W)$.
Since $W' = W_2W_3(3)$, we have $\sigma(W') > \sigma(W)$.
\end{enumerate}

The proof for $r=3$ is complete.
The general case $r \geq 3$ follows from the existence of the walk $\Twalk{Z}{L^n}{L^s}{Z}$.
\end{proof}

\begin{example}
For instance, for the edge $\Tedge{M}{L^2R}{RL^4}{N}$, where $M = \mat{6}{2}{2}{3}$ and $N = \mat{2}{1}{0}{7}$ in the transducer $\mathcal{T}_{14}$, we have $\Tedge{M}{L^4}{LR^2}{\Tedge{A_{14}}{L^{14}}{L}{A_{14}}}$ and $\Tedge{A_{14}}{L^{12}R}{RL^6}{N}$ with $A_{14} = \mat{14}{0}{0}{1}$. Therefore, using the notation from the last lemma, we have $i = 2, j = 1$, $t = 2$ and there is a walk $\Twalk{A_{14}}{L^{rn-2}R}{L^{r-1}RL^6}{M}$.
It means that $W = RL^4$ and $W' = L^{r-1}RL^6$ and so for all $r\geq 3$, $\sigma(W') = 3 > 2 = \sigma(W)$ and moreover, $R^{-1}W = L^4$ is a suffix of $W'$.
\end{example}

\begin{corollary} \label{coro:kappa_hrana}
Let $\Tedge{M}{VR^j}{W}{N}$ with $j > 0$. 
If $V$ is empty, then the walk
\begin{enumerate}[(a)]
\item $\Twalk{M}{L^{4n}R^j}{\widehat{W}}{N}$ with $\sigma(\widehat{W}) > \sigma(W)$ \label{it:kappa_hrana_prazdneV}
\end{enumerate}
exists.
If $V$ ends in $L$, then the walks
\begin{enumerate}[resume*]
  \item $\Twalk{M}{\kappa(V)R^j}{\widehat{W}}{N}$ with $\sigma(\widehat{W}) > \sigma(W)$, \label{it:kappa_hrana_nic}
  \item $\Twalk{M}{L^{4n}\kappa(V)R^j}{\widehat{W}_L}{N}$ with $\sigma(\widehat{W}_L) > \sigma(W)$, and \label{it:kappa_hrana_L}
  \item $\Twalk{M}{R^{4n}\kappa(V)R^j}{\widehat{W}_R}{N}$ with $\sigma(\widehat{W}_R) > \sigma(W)$ \label{it:kappa_hrana_R}
\end{enumerate}
exist.
\end{corollary}

\begin{proof}
We shall prove the existence of the walk~\ref{it:kappa_hrana_prazdneV} directly and existence of the other walks by induction on $\sigma(V)$.

Assume that $V = L^i$ with $i \geq 0$, i.e., $\sigma(V) \leq 1$.
Using \Cref{le:nafouknuti_jedno} with $r=4$, we obtain
\[
\Tedge{M}{L^{i+t}}{W_1}{Q} \quad \text{ and } \quad \Twalk{Q}{L^{4n-t}R^j}{W'}{N}.
\]
with $\sigma(W') > \sigma(W)$ for some $Q \in \DB{}$, $W_1, W' \in \Dj$ and an integer $t$ with $0 < t \leq n$.
Therefore, the walk $\Twalk{M}{L^{i+4n}R^j}{W_1W'}{N}$ exists.
Applying \Cref{le:vyfouknuti_n} the correct number of times, we obtain $\Twalk{M}{\kappa(L^i)R^j}{\widehat{W}}{N}$ with $\sigma(\widehat{W}) = \sigma(W_1W')\geq \sigma(W') > \sigma(W)$.
This proves existence of \ref{it:kappa_hrana_prazdneV} and \ref{it:kappa_hrana_nic} for $\sigma(V) =1 $.
Using the symmetric version of \Cref{le:nafouknuti_jedno} (using the symmetry of $L$ and $R$ given by \Cref{prop:assoc_sym}) on $\Tedge{M}{L^{i+t}}{W_1}{Q}$, we obtain that the walk
\[
\Twalk{M}{R^{4n}L^{i+t}}{W_2}{Q}
\]
exists.
Therefore, $\Twalk{M}{R^{4n}L^{4n+i}R^j}{W_2W'}{N}$ exists.
Considering again \Cref{le:vyfouknuti_n}, we prove \ref{it:kappa_hrana_R}.

To show the existence of \ref{it:kappa_hrana_L}, we proceed as in the case of the walk~\ref{it:kappa_hrana_nic} except for using \Cref{le:nafouknuti_jedno} with $r=8$ and factoring $L^{4n}$ in the input word of the obtained walk.

By \Cref{prop:assoc_sym}, the symmetric version of the claim for $\sigma(V) \leq 1$ holds.

Assume now the claim and its symmetric version hold for $\sigma(V') = k$ and let $V = V'L^i$ with $\sigma(V'L^i) = k+1$.
We apply \Cref{le:nafouknuti_jedno} on $\Tedge{M}{V'L^iR^j}{W}{N}$: there exist $Q \in \DB{}$, $W_1, W' \in \Dj$ and an integer $t$ with $0 < t \leq n$ such that
\[
\Tedge{M}{V'L^{i+t}}{W_1}{Q} \quad \text{ and } \quad \Twalk{Q}{L^{4n-t}R^j}{W'}{N}
\]
with $\sigma(W') > \sigma(W)$.
By the induction hypothesis and the symmetry of $L$ and $R$ (\Cref{prop:assoc_sym}) on $\Tedge{M}{V'L^{i+t}}{W_1}{Q}$, we obtain $\Twalk{M}{X\kappa(V')L^{i+t}}{W_1'}{Q}$ with $\sigma(W_1') > \sigma(W_1)$ and $X$ being empty, $L^n$, or $R^n$.
The situation is illustrated in \Cref{fig:nafukovani}.
Therefore,
\[
\Twalk{M}{X\kappa(V')L^{i+4n}R^j}{W_1'W'}{N}.
\]
Using \Cref{le:vyfouknuti_n}, we obtain $\Twalk{M}{X\kappa(V'L^i)R^j}{\widehat{W}}{N}$ with $\sigma(\widehat{W}) = \sigma(W_1'W') \geq \sigma(W') > \sigma(W)$.
\end{proof}

\begin{remark}
In fact, the mapping $\kappa$ could be defined such that it adjusts the length of each run between $3n$ and $(4n-1)$ and we could prove the same bound in \Cref{thm:main}.
Our experiments show that it might also be sufficient to adjust the length of all runs between $2n$ and $(3n-1)$.
However, we use the given definition of $\kappa$ since it simplifies the proofs without changing the result.
\end{remark}

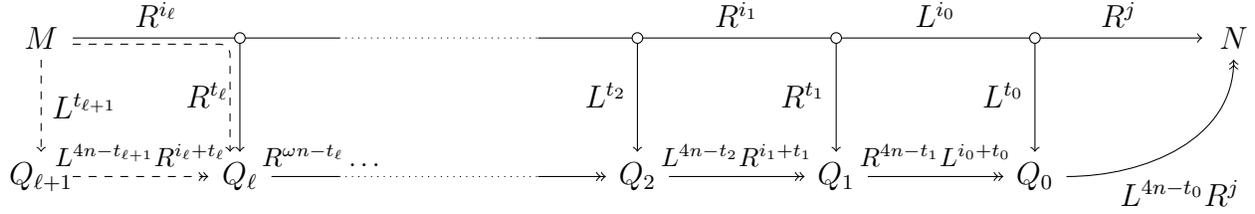
\begin{figure}
\centering
\begin{tikzpicture}[small/.style={fontscale=-1}]
  \matrix (m) [matrix of math nodes,row sep=3em,column sep=4.4em,minimum width=2em]
  {
     M  & \phantom{N}  &  \phantom{N} &  \phantom{N} & \phantom{N} & \phantom{N}  & N \\
     \phantom{Q}   &  Q_{\ell} & \phantom{Q} & Q_2 & Q_1  &  Q_0 &  \\};
\path (m-1-6.center) edge[->] node[above]{$R^j$} (m-1-7);
\path (m-1-5.center) edge[-] node[above]{$L^{i_0}$} (m-1-6.center);
\path (m-1-4.center) edge[-] node[above]{$R^{i_1}$} (m-1-5.center);
\draw[] (m-1-2.center) edge[-] ($(m-1-2)!0.5!(m-1-3)$) ($(m-1-2)!0.5!(m-1-3)$) edge [dotted] ($(m-1-3)!0.5!(m-1-4)$) ($(m-1-3)!0.5!(m-1-4)$) edge[]  (m-1-4.center);
\path (m-1-1) edge[-] node[above]{$R^{i_{\ell}}$} (m-1-2.center);

\path (m-1-6.center) edge[->] node[left]{$L^{t_0}$} (m-2-6);
\path (m-1-5.center) edge[->] node[left]{$R^{t_1}$} (m-2-5);
\path (m-1-4.center) edge[->] node[left]{$L^{t_2}$} (m-2-4);

\path (m-1-2.center) edge[->] node[left]{$R^{t_\ell}$} (m-2-2);

\path (m-2-6) edge[->>,out=0,in=-90] node[below,yshift=-0.7em]{$L^{4n-t_0}R^j$} (m-1-7);
\path (m-2-5) edge[->>] node[above,small]{$R^{4n-t_1}L^{i_0+t_0}$} (m-2-6);
\path (m-2-4) edge[->>] node[above,small]{$L^{4n-t_2}R^{i_1+t_1}$} (m-2-5);
\draw[] (m-2-2) edge[-] node[above,small,xshift=0.5em]{$R^{\omega n - t_\ell}\ldots$} ($(m-2-2)!0.5!(m-2-3)$) ($(m-2-2)!0.5!(m-2-3)$) edge [dotted] ($(m-2-3)!0.5!(m-2-4)$) ($(m-2-3)!0.5!(m-2-4)$) edge[->>]  (m-2-4);


\node at (m-2-1){$Q_{\ell+1}$};
\path (m-1-1) edge[->,dashed] node[right]{$L^{t_{\ell+1}}$} (m-2-1);
\path (m-2-1) edge[->>,dashed] node[above,small]{$L^{4n-t_{\ell+1}}R^{i_\ell+t_\ell}$} (m-2-2);

\draw[rounded corners,dashed] ([yshift=0.5em]m-1-1.south east) -- ([xshift=0.7em,yshift=0.5em]m-1-2.south west) [->] --  ([xshift=0.7em]m-2-2.north west);


\filldraw[fill=white] (m-1-6) circle (2pt);
\filldraw[fill=white] (m-1-5) circle (2pt);
\filldraw[fill=white] (m-1-4) circle (2pt);
\filldraw[fill=white] (m-1-2) circle (2pt);

\end{tikzpicture}
\caption{
An illustration of the idea in the first part of the proof of \Cref{coro:kappa_hrana}.
The figure should be read as follows: the top line contains the original edge $\Tedge{M}{VR^j}{W}{N}$ with $\sigma(V) = \ell+1$ even, and we proceed from right to left by constructing a new walk going through the vertices $Q$ given by \Cref{le:nafouknuti_jedno}.
We start by finding $Q_0$ and find a walk with the length of the penultimate run in the input word modified.
We continue from right to left until we reach the first run, indexed by $\ell$ in the figure, in the current input word.
The last steps depends on which item is being shown: \Cref{it:kappa_hrana_L} is depicted by the bottom dashed path on the left with $\omega=4$,
\Cref{it:kappa_hrana_nic} is the top dashed path on the left with $\omega=4$, and finally \Cref{it:kappa_hrana_R} is the top dashed path on the left with $\omega=8$.
The 3 cases correspond to $X$ being $L^n$, empty or $R^n$, respectively.
A figure for $\sigma(V)$ odd is analogous.
}
\label{fig:nafukovani}
\end{figure}

The following lemma describes the paths in the transducer $\T{}$ that are taken when the greatest prolongation occurs.

\begin{lemma} \label{le:LS_n_do_RS_n}
Let $M = \mat{t_1}{0}{u_1}{m_1} \in \LE{}$ and $i \in \{\LclassMax{M}, \LclassMax{M} + 1, \dots, 2 \LclassMax{M}-1\}$.
\begin{enumerate}[(1)]
  \item There exist one and only one matrix $N_{L,M,i} \in \RE{}$ such that $\Twalk{M}{L^iR^{j}}{W}{N_{L,M,i}}$ for some $j \in \N$ and $W \in \Dj$.
  Moreover, there is exactly one $j_{L,M,i} \in \{3n - \RclassMax{N} + 1, 3n - \RclassMax{N} + 2, \dots, 3n\}$ such that the walk from $M$ to $N_{L,M,i}$ with the input word $L^iR^{j_{L,M,i}}$ and an output word $W_{L,M,i}$ exists.
  $N_{L,M,i}, j_{L,M,i}$ and $W_{L,M,i}$ depend only on $M$ and $i$. \label{le:LS_n_do_RS_n:i1}
  \item The word $W_{L,M,i}$ starts with $L$, ends with $R$ and satisfies
\[
\sigma(W_{L,M,i}) = 2\left\lfloor \frac{\xi_{L,M,i}}{2} \right\rfloor +2,
\]
where $\xi_{L,M,i} = \xi({im_1+u_1}, {t_1})$. \label{le:LS_n_do_RS_n:i2}
\end{enumerate}
\end{lemma}

\begin{proof}
Assume that we are on a walk which starts in $M$, we input $L^i$, and we start inputting a run of $R$'s.
\Cref{le:max_n_jednoho_pismene,le:spadnu_do_LS} imply that after inputting at most $2n$ $R$'s we reach a state $N \in \RE{}$.
Moreover, the matrix $N$ is unique and depends only on $M$ and $i$ (and the fact that the walk started by reading $L$'s).
By \Cref{def:nu}, there is exactly one $j \in \{3n - \RclassMax{N} + 1, 3n - \RclassMax{N} + 2, \dots, 3n\}$ such that $\Twalk{M}{L^iR^{j}}{W}{N}$, depending only on $M$ and $i$.
Therefore, there is also exactly one word $W$, and it depends only on $M$ and $i$, which concludes the proof of \Cref{le:LS_n_do_RS_n:i1}.

Let $N_{L,M,i} = \mat{t_2}{u_2}{0}{m_2}$.
We have
\[
W_{L,M,i} = ML^iR^{j_{L,M,i}}N_{L,M,i}^{-1} = \mat{\frac{t_1}{t_2}}{f}{\frac{im_1+u_1}{t_2}}{e}.
\]
for some $e,f \in \N$.
Using $W_{L,M,i} \in \Dj$, we conclude that $\frac{im_1+u_1}{t_2}, \frac{t_1}{t_2} \in \N$ and $\gcd{ (\frac{im_1+u_1}{t_2}, \frac{t_1}{t_2})} = 1$.
Therefore, $t_2 = \gcd{ ((im_1 + u_1),t_1)}$.
As $i \geq \LclassMax{M}$, the walk $\Twalk{M}{L^iR^{j_{L,M,i}}}{W_{L,M,i}}{N_{L,M,i}}$ starts with the walk $\Twalk{M}{L^{\nu}}{L^t}{M}$ with $\nu = \LclassMax{M}$ and some $t>0$, and the walk ends with the walk $\Twalk{N}{R^n}{R^s}{N}$.
Thus, the output word of the first walk is a power of $L$, and $L$ is the first letter of $W_{L,M,i}$ and the output word of the last walk is a power of $R$, and $R$ is the last letter of $W_{L,M,i}$.

By \Cref{le:pocet_LRzmen}, $\sigma(W_{L,M,i}) =  2 \left\lfloor \frac{\xi \left(\frac{im_1+u_1}{t_2}, \frac{t_1}{t_2}\right)}{2} \right\rfloor + 2$.
By definition of $\xi$, we have $\xi(\frac{im_1+u_1}{t_2}, \frac{t_1}{t_2}) = \xi(im_1+u_1,t_1) = \xi_{L,M,i}$.
\end{proof}

In what follows, the notation introduced by the last claim is used.
We also use the symmetric version of this notation in the following sense.
The symmetric version of~\Cref{le:LS_n_do_RS_n} holds, and given $M \in \RE{}$ and $i \in \{\RclassMax{M}, \RclassMax{M} + 1, \dots, 2 \RclassMax{M}-1\}$, we find $N_{R,M,i} \in \LE{}$, $j_{R,M,i} \in \{3n - \LclassMax{N_{R,M,i}} + 1, 3n - \LclassMax{N_{R,M,i}} + 2, \dots, 3n\}$ and $W_{R,M,i} \in \Dj$.
The $L$, resp. $R$, in the subscript is needed for the case $M \in \LS{} \cap \RS{} \neq \emptyset$.
In general, we may have $j_{L,M,i} \neq j_{R,M,i}$.
Similarly, we also use the notation $\xi_{R,M,i}$.

\subsection{Closed walks}

In this section, we return to the computation of the given  M\"{o}bius transformation for a periodic input word.
Clearly, the computation in the transducer $\T{}$ ends in a repeating loop, i.e., we end up with some closed walk $\Twalk{M}{V}{W}{M}$.
We start with a general lemma that deals with the case when this closed walk is symmetric in the following sense.
If the closed walk $\Twalk{M}{V}{W}{M}$ can be decomposed into two parts where the first part is $\Twalk{M}{V_1}{W_1}{\assoc{M}}$ and the second part is $\Twalk{\assoc{M}}{\assoc{V_1}}{W_2}{M}$, then we say that the closed walk is \emph{symmetric}.
Note that this implies that $V = V_1\assoc{V_1}$, and by \Cref{prop:assoc_sym} it also implies $W_2 = \assoc{W_1}$, which is stated as the following lemma.

\begin{lemma} \label{le:sym_dvojstav}
If $\Twalk{M}{V}{W}{M}$ is a symmetric closed walk, then $W = W_0\assoc{W_0}$ for some $W_0$.
\end{lemma}

Note that symmetricity of the closed walk $\Twalk{M}{V}{W}{M}$ is not equivalent to $V = V_1\assoc{V_1}$ and $W = W_1\assoc{W_1}$.

We introduce the following notation, which is due to the fact that we shall require to change the starting vertex of a closed walk.
Let $V \in \{L,R\}^*$.
We set $\tau(V)$ to be the set of all conjugate words of $V$.
A word $W$ is \emph{conjugate} to a word $V$ if $V = V_1V_2$ and $W = V_2V_1$, i.e., $V$ is cyclic shift of $W$.
Furthermore, we set
\[
\sigmaC(V) = \min \left\{ \sigma(V') \colon V' \in \tau(V) \right\},
\]
i.e., except for the case $\sigma(V) = 1$, the mapping $\sigmaC$ counts the number of runs of a conjugate word of $V$ which starts and ends in a distinct letter, that is $\sigmaC(V) = 2 \left \lfloor \frac{\sigma(V)}{2} \right \rfloor$.
Note that if $\sigma(V) \geq \sigma(W)$, then $\sigmaC(V) \geq \sigmaC(W)$.

Let $W \in \Dj$ be non-empty. We set
\[
\tau_\kappa(W) = \tau(\kappa(W')) \quad \text{ where } W' \in \tau(W), \sigma(W')=\sigmaC(W).
\]
The mapping $\tau_\kappa$ defines the transformation of the input word which yields the upper bound in \Cref{thm:main}.
The following theorem expresses its key role.

\begin{restatable}{theorem}{restatableVyfukovani}
\label{co:vyfukovani_ze_symetrickeho}
Let $\Twalk{M}{V}{W}{M}$.
There exists $\widehat{V} \in \tau_\kappa(V)$ such that $\Twalk{\widehat{M}}{\widehat{V}}{\widehat{W}}{\widehat{M}}$ for some $\widehat{M} \in \DB{}$ with $\sigmaC(\widehat{W}) \geq \sigmaC(W)$.
Moreover,
\begin{enumerate}[(a)]
\item \label{it:1)vyfukovani}
if $V = V_1^{m_1}$ for some $V_1, m_1 \geq 2$ and $\Twalk{\widehat{M}}{\widehat{V}}{\widehat{W}}{\widehat{M}} = \left( \Twalk{\widehat{M}}{\widehat{V_1}}{\widehat{W_1}}{\widehat{M}} \right)^{m_1}$ where $\widehat{W_1}^{m_1} = \widehat{W}$ and $\widehat{V_1}^{m_1} = \widehat{V}$, then $\Twalk{M}{V}{W}{M}=(\Twalk{M}{V_2}{W_2}{M})^{m_2}$ for some $m_2 \geq 2$ and $W_2^{m_2} = W$ and $V_2^{m_2} = V$;
\item \label{it:2)vyfukovani}
if $V = V_1\assoc{V_1}$ for some $V_1$ and $\Twalk{\widehat{M}}{\widehat{V}}{\widehat{W}}{\widehat{M}}$ is symmetric, then $\Twalk{M}{V}{W}{M}$ is symmetric.
\end{enumerate}
\end{restatable}

\subsubsection{Proof of \Cref{co:vyfukovani_ze_symetrickeho}}
\label{subsec:kappa_walks}

Several technical lemmas concerning possible forms of edges in the transducer $\T{}$ follow.
These lemmas are used for the proof of \Cref{co:vyfukovani_ze_symetrickeho} given at the end of this subsection.

\begin{lemma} \label{le:vstup_V_1}
Let $\Tedge{M}{V}{W}{P}$, where $M,P \in \DB{}$, $V, W \in \Dj$.
\begin{enumerate}
  \item \label{le:vstup_RLk_1}
If $V = \widehat{V}RL^{j}$ for some $ \widehat{V} \in \Dj$ and $j\geq 2$, then
\[
W = LR^{i} \quad \text{ and } \quad i \geq j-1.
\]
\item \label{le:vstup_LRjL_1}
If $V = \widehat{V}LR^{j}L$ for some $ \widehat{V} \in \Dj$ and $j\geq 1$, then
\[
W = LR^{i} \quad \text{ and } \quad i \geq 1.
\]
\item \label{le:vstup_L_1}
IF $V = L^\ell$ with $\ell \geq 1$, then
\[
W = LR^s \quad \text{ or } \quad W = L^t
\]
where $s \geq 0, t \geq 2$. Moreover, if $W = L^t$, then $\ell = 1$.
\end{enumerate}
\end{lemma}

\begin{proof}
We start with the proof of \Cref{le:vstup_RLk_1}.

Let $M \widehat{V} = \mat{a}{b}{c}{d}$. Then
\[
M \widehat{V}RL^{j} = \mat{(j+1)a+jb}{a+b}{(j+1)c+jd}{c+d} \not \in \RB{}.
\]
Moreover, by \Cref{le:stejna_pismena}, the word $W$ has prefix $L$.
Therefore, $c+d>a+b$.
Since
\[
M \widehat{V}RL^{k} = \mat{(k+1)a+kb}{a+b}{(k+1)c+kd}{c+d}  \in \RB{}
\]
for all $k <j$, we have

\begin{equation} \label{eq:k_RLk}
(k+1)a+kb >(k+1)c+kd.
\end{equation}
For $k = j-1$ we obtain
\begin{equation} \label{eq:1_RLk}
ja+(j-1)b >jc+(j-1)d.
\end{equation}
In particular, since $j \geq 2$, we obtain for $k=1$:
\begin{equation} \label{eq:2_RLk}
2a+b >2c + d.
\end{equation}

If $W$ has prefix $L^2$, then
\[
L^{-2}WP = \mat{(j+1)a+jb}{a+b}{(j+1)(c-2a)+j(d-2b)}{c+d-2(a+b)} \in \D{}
\]
and therefore
\[
0 \leq(j+1)(c-2a)+j(d-2b)  \overset{\eqref{eq:1_RLk}}{<} - ja-2a -(j+1)b+c+d \overset{\eqref{eq:2_RLk}}{<} 0
\]
which is a contradiction. Therefore, $L^2$ is not a prefix of $W$. Let $i$ be maximal possible such that $LR^i$ is a prefix of $W$. We have
\[
R^{-i}L^{-1}WP = R^{-i}L^{-1}M \widehat{V}RL^{j} =
\]
\[
= \mat{(i+1)[(j+1)a+jb] - i [(j+1)c+jd]}{(i+1)(a+b)-i(c+d)}{j(c+d) - j(a+b)+c-a}{c+d-a-b} \in \D{}.
\]
It follows that
\begin{equation} \label{eq:b_RLk}
(i+1)(a+b) -i(c+d)\geq 0.
\end{equation}

If $j > i+1$, then \Cref{eq:k_RLk} holds for $k = i+1$ and therefore

\begin{equation} \label{eq:i1_RLk}
(i+2)a+(i+1)b >(i+2)c+(i+1)d.
\end{equation}

Whereas using that $i$ is maximal possible, we have
\[
(R^{-i}L^{-1}WP)_{1,2} \leq (R^{-i}L^{-1}WP)_{2,2} \iff (i+1)(a+b)-i(c+d) \leq c+d-a-b
\]
\[
\implies
\]
\[
0 \leq (i+1)(c+d) - (i+2)(a+b) \overset{\eqref{eq:i1_RLk}}{<} -b -c \leq 0
\]
which is a contradiction.
It means that $W$ has prefix $LR^i$ for some $i \geq j-1$.

We show that $W$ is also equal to this prefix.
Again, we suppose for contradiction that $W$ has prefix $LR^{i}L$.
It means that
\[
(R^{-i}L^{-1}WP)_{1,1} \leq (R^{-i}L^{-1}WP)_{2,1} \iff
(i+1)[(j+1)a+jb] - i [(j+1)c+jd] \leq j(c+d) - j(a+b)+c-a
\]
\[
\implies
\]
\[
0 \geq (i+2)[(j+1)a+jb] - (i+1) [(j+1)c+jd]  \overset{\eqref{eq:1_RLk}}{>} (i+2)(a+b)-(i+1)(c+d)+(ja+(j-1)b) \overset{\eqref{eq:b_RLk}}{\geq}
\]
\[
\overset{\eqref{eq:b_RLk}}{\geq} a+b-(c+d)+(ja+(j-1)b) \overset{\eqref{eq:2_RLk}}{>} 0
\]
where the last inequality uses that $j \geq 2$. This is a contradiction and therefore $W = LR^i$ and the first part of this lemma is proven.


We continue with the proof of \Cref{le:vstup_LRjL_1}.

Let again $M\widehat{V} = \mat{a}{b}{c}{d}$. We have
\[
M\widehat{V}LR^j = \mat{a+b}{ja+(j+1)b}{c+d}{jc+(j+1)d} \in \RB{}.
\]
It follows that
\begin{equation} \label{eq:ab_LRjL}
a+b>c+d.
\end{equation}
By \Cref{le:stejna_pismena}, we know that $W$ has prefix $L$ and thus
\[
L^{-1}WP = L^{-1} M\widehat{V}LR^jL =
\]
\[
= \mat{(j+1)a+(j+2)b}{ja+(j+1)b}{(j+1)c + (j+2)b-(j+1)a-(j+2)b}{jc+(j+1)d-ja-(j+1)b} \in \D{}.
\]
Moreover, we have
\[
(L^{-1}WP)_{1,2} - (L^{-1}WP)_{2,2} =
2ja+2(j+1)b-jc-(j+1)d \geq (j+1)(a+b-c-d) \overset{\eqref{eq:ab_LRjL}}{>} 0,
\]
which means that $W$ has prefix $LR$. Let $i\geq 1$ be maximal possible such that $W$ has prefix $LR^i$. We have
\[
R^{-i}L^{-1}WP = R^{-i}L^{-1}M \widehat{V}LR^jL =
\]
\[
 \mat{(i+1)[(j+1)a+(j+2)b] - i[(j+1)c+(j+2)d]}{(i+1)[ja+(j+1)b]-i[jc+(j+1)d]}{(j+1)c+(j+2)d-(j+1)a-(j+2)b}{jc+(j+1)d-ja-(j+1)b} \in \D{}
\]
and therefore
\begin{equation} \label{eq:i_LRjL}
(i+1)(ja+(j+1)b)-i(jc+(j+1)d) \geq 0.
\end{equation}
Therefore,
\[
(R^{-i}L^{-1}WP)_{1,1} - (R^{-i}L^{-1}WP)_{2,1} =
(i+2)[(j+1)a+(j+2)b] - (i+1)[(j+1)c+(j+2)d] \overset{\eqref{eq:i_LRjL}}{\geq}
\]
\[
 \overset{\eqref{eq:i_LRjL}}{\geq} (j+i+2)a+(j+i+3)b- [(j+i+1)c + (j+i+2)d]
\geq (j+i+2)(a+b-c-d)  \overset{\eqref{eq:ab_LRjL}}{>} 0
\]
and finally $W = LR^i$.

To prove the last item of the lemma, let $M = \mata$.

We have
\[
ML^k = \mat{a+kb}{b}{c+kd}{d} \in \RB{}
\]
for all $k <\ell$, which means that $a+kb>c+kd$ and in particular
\begin{equation} \label{eq:l1_vstup_L}
a+(\ell-1)b>c+(\ell-1)d.
\end{equation}
By \Cref{le:stejna_pismena}, we know that $W$ has prefix $L$.
We first investigate the case in which $L^2$ is not a prefix of $W$. We take $s \geq 0$ maximal possible such that $LR^s$ is a prefix of $W$. If $s = 0$, we have $W = L$ and the claim holds. In the case $s \geq 1$, we have
\[
R^{-j}L^{-1}WP = R^{-j}L^{-1}ML^{\ell} = \mat{(a+\ell b)(j+1)-j(c+\ell d)}{(j+1)b-jd}{c+\ell d-(a+\ell b)}{d-b} \in \D{}
\]
for all $j \leq s$.
Therefore, $(j+1)b-jd \geq 0$ and specially for $j = 1$ ($s \geq 1$)
\begin{equation} \label{eq:1_vstup_L}
2b-d \geq 0
\end{equation}
and for $j = s$
\begin{equation} \label{eq:s1_vstup_L}
(s+1)b-sd \geq 0.
\end{equation}
It follows that
\[
(R^{-s}L^{-1}WP)_{1,1} - (R^{-s}L^{-1}WP)_{2,1} =
(a+\ell b)(s+2)-(s+1)(c+\ell d) =
\]
\[
 = (s+1)(a-c+(\ell-1)(b-d))+a+\ell b+(s+1)b-(s+1)d \overset{\eqref{eq:s1_vstup_L}}{\geq}(s+1)(a-c+(\ell -1)(b-d)) + a+\ell b-d \overset{\eqref{eq:l1_vstup_L}}{>}
\]
\[
\overset{\eqref{eq:l1_vstup_L}}{>} a+\ell b-d > 2b-d \overset{\eqref{eq:1_vstup_L}}{\geq} 0.
\]
Therefore, we have
\[
(R^{-s}L^{-1}WP)_{1,1} > (R^{-s}L^{-1}WP)_{2,1}
\]
and therefore the word $W$ cannot have prefix $LR^sL$.
Thus, $W = LR^s$.

Now, we investigate if $W$ can have prefix $L^2$.
We have
\[
L^{-1}WP = L^{-1}ML^\ell  = \mat{a+\ell b}{b}{c+\ell d-(a+\ell b)}{d-b}.
\]
and for $\ell  \geq 2$ we have
\[
(L^{-1}WP)_{1,1} - (L^{-1}WP)_{2,1} =
2a+2\ell b-c-\ell d \overset{\eqref{eq:l1_vstup_L}}{>} 2b + c+ (\ell -2)d \geq 0
\]
and therefore $W$ cannot have the prefix $L^2$.
It remains to deal with the case $\ell  = 1$ and $W$ has prefix $L^t$ for some $t \geq 2$ maximal possible. We have
\[
L^{-t}WP = L^{-t}ML^ = \mat{a+b}{b}{c+d-t(a+b)}{d-tb} \in \D{}.
\]
Therefore, we have
\begin{equation} \label{eq:t_vstup_L}
(c+d)-t(a+b)\geq 0.
\end{equation}
Further,
\[
(L^{-t}WP)_{2,2} - (L^{-t}WP)_{1,2} =
d-(t+1)b \overset{\eqref{eq:t_vstup_L}}{\geq} ta-c-b \geq 2a-c-b >0
\]
where we have used $t \geq 2$.
Therefore, $W$ cannot have prefix $L^tR$, which means $W = L^t$.
\end{proof}

\begin{corollary} \label{le:vstup_V_2}
Let $\Tedge{M}{V}{W}{P}$ with $M,P \in \DB{}$, $V, W \in \Dj$.
\begin{enumerate}
  \item \label{le:vstup_RLk_2}
If $W = L^{j}R\widehat{W}$ for some $ \widehat{W} \in \Dj$ and $j\geq 2$, then
\[
V = R^{i}L \quad \text{ and } \quad i \geq j-1.
\]
\item \label{le:vstup_L_2}
If $W = L^{\ell}$ and $\ell \geq 1$, then
\[
V = R^{s}L \quad \text{ or } \quad V = L^{t},
\]
where $s \geq 0, t \geq 2$. Moreover, if $V = L^{t}$, then $\ell = 1$.
\end{enumerate}
\end{corollary}

\begin{proof}
The two claims follow directly from \Cref{le:vstup_RLk_1,le:vstup_L_1} of \Cref{le:vstup_V_1} and \Cref{prop:assoc_sym}, namely the fact that $\Tedge{M}{V}{W}{P}$ implies $\Tedge{P^T}{W^T}{V^T}{M^T}$.
\end{proof}

\begin{lemma} \label{le:hrany_v_cyklu}
Every closed walk in  the transducer $\T{}$ either includes at least one edge with a nonempty input which cannot be written as $L^{q}R$ or $R^{q}$ for some $q \geq 1$ or the input word of this walk is $R^r$ for some $r \geq 1$.
\end{lemma}
\begin{proof}
Let $M = \mat{a}{b}{c}{d}$, $N = \mat{e}{f}{g}{h}$ and $\Tedge{M}{L^qR}{W}{N}$. By \Cref{le:stejna_pismena}, $W$ has prefix $R$. By direct computation, we obtain
\[
(R^{-1}WN)_{1,1} = a+qb - c - qd
\]
and therefore
\[
e \leq a+ qb-c-qd \leq a + q(b-d) < a,
\]
where the last inequality follows from $M \in \DB{}$ which implies $b<d$.

Similarly for the transition $\Tedge{M}{R^q}{W}{N}$, we obtain using \Cref{le:vstup_L_1} of \Cref{le:vstup_V_1} that either $W = RL^s$ for some $s \geq 0$, or $W = R^t$ for $t \geq 2$ and $q= 1$. By direct computation, we obtain
\[
e = a-c \leq a
\]
in the first case and
\[
e = a-tc \leq a
\]
in the second case.

This means that the number in the first row and first column of the state matrix after taking the edge with the input word $R^q$ is either the same or it is lessened, and after taking the edge with the input word $L^qR$ it is always lessened.
\end{proof}

\begin{lemma} \label{le:hrany_do_P}
Let $\Tedge{N}{L^\ell }{LR^s}{P}$, where $N,P \in \DB{}$, $N = \mat{a}{b}{c}{d}$, $\ell  \geq 1$ and $s \geq \ell +1$. We have:
\[
a <2b.
\]
\end{lemma}
\begin{proof}
Let $P = \mat{e}{f}{g}{h}$. By direct computation, we obtain
\[
N = LR^sPL^{-\ell } = \mat{e+sg-\ell (f+sh)}{f+sh}{e+(s+1)g-\ell (f+(s+1)h)}{f+(s+1)h} \in \DB{}.
\]

Therefore,
\begin{equation} \label{eq:cd_hrany_do_P}
f+(s+1)h>e+(s+1)g-\ell (f+(s+1)h)
\end{equation}
and we have
\[
2b-a = 2(f+sh) + \ell (f+sh)-e-sg \overset{\eqref{eq:cd_hrany_do_P}}{>} f+(s-\ell -1)h+g \geq 0
\]
where the last inequality holds because $s\geq \ell +1$. It follows that $2b>a$.
\end{proof}

\begin{lemma} \label{le:vstup_do_N}
Let $\Tedge{M}{V}{W}{N}$, where $N,M \in \DB{}$, $N = \mat{a}{b}{c}{d}$, $V,W \in \Dj$, $V$ has suffix $L$ and $a\leq 2b$. We have:
\[
V = R^pL,
\]
where $p \geq 1$.
\end{lemma}
\begin{proof}

By \Cref{le:stejna_pismena}, we know that since $V$ has suffix $L$, the word $W$ has prefix $L$. We have by \Cref{le:vstup_V_1} that $W = LR^k$ for some $k\geq 1$ or $W = L^s$ for some $s \geq 1$ or $V = R^pL$ for some $p\geq 1$.
In the third case the claim holds. The other two possibilities are discussed separately.
\begin{enumerate}
\item $W = LR^k$ for some $k \geq 1$.

According to \Cref{prop:assoc_sym}, there exists the edge $\Tedge{N^T}{W^T}{V^T}{M^T}$. The starting vertex of this edge is $N^T = \mat{a}{c}{b}{d}$ and the input word is $W^T = L^kR$. Using \Cref{thm:Raney_edge} we can determine the output word of this edge (the word $V^T$). We already know that the word $V^T$ has prefix $R$.

We have
\[
R^{-1}N^TL^kR = \mat{a+kc-b-kd}{a+(k+1)c-b-(k+1)d}{b+kd}{b+(k+1)d}.
\]
Moreover, because $a \leq 2b$ and $c<d$, we have
\[
(R^{-1}N^TL^kR)_{1,1} - (R^{-1}N^TL^kR)_{2,1} =
a-2b+k(c-2d) < 0
\]
which means that $V^T$ has prefix $RL$.

Let $p \geq 1$ be maximal possible such that $V^T$ has prefix $RL^p$. We have
\[
L^{-p}R^{-1}N^TW^T = \mat{a+kc-b-kd}{a+(k+1)c-b-(k+1)d}{(p+1)(b+kd)-p(a+kc)}{(p+1)(b+(k+1)d)-p(a+(k+1)c)} \in \D{}
\]
and therefore
\begin{equation} \label{eq:p_vstup_do_N}
(p+1)(b+kd)-p(a+kc) \geq 0.
\end{equation}
Moreover,
\[
(L^{-p}R^{-1}N^TW^T )_{2,2} - (L^{-p}R^{-1}N^TW^T )_{1,2} =
(p+2)(b+(k+1)d)-(p+1)(a+(k+1)c) \overset{\eqref{eq:p_vstup_do_N}}{\geq}
\]
\[
\overset{\eqref{eq:p_vstup_do_N}}{\geq}
b+(k+1)d-(a+(k+1)c) +(p+1)d-pc =
(k+1)(d-c)+p(d-c)+(b+d-a) >
\]
\[
> b+d-a >2b-a \geq 0
\]
where the last inequality follows from $a \leq 2b$. Therefore, $V^T = RL^p$.

\item $W = L^s$ for some $s \geq 1$.

Because the input word has suffix $L$ we know by \Cref{le:vstup_V_1} that $V = R^pL$ for some $p \geq 1$ or $V = L^t$ for some $t\geq 1$.
Now, we suppose for contradiction that the second case holds. We have $M = WNV^{-1} = L^s N L^{-t} =  \mat{a-tb}{b}{s(a-tb)+c-td}{sb+d}$. Moreover $M \in \DB{}$ and therefore
\[
M_{1,1} >M_{1,2} \iff a-tb>b \implies a>2b
\]

which is a contradiction.

\end{enumerate}

\end{proof}

A few lemmas, which are used in the proof of \Cref{prop:hrany_do_P2}, follow.

\begin{lemma} \label{le:1_hrany_do_P2}
Let
\[
\Tedge{P_1}{V_1L^{k_1}}{LR^{j_1}}{M} \text{ and } \Tedge{P_2}{V_2L^{k_2}}{LR^{j_2}}{M},
\]
where $P_1, P_2,M \in \DB{}$, $V_1,V_2 \in \Dj$ and neither of them has suffix $L$, $k_1,k_2 \geq 1,j_1,j_2 \geq 0$ and $j_2 >j_1$.
Then $k_2 \leq k_1$.
\end{lemma}
\begin{proof}
We put $M = \mat{a}{b}{c}{d}$. Then
\[
P_1V_1 = LR^{j_1}ML^{-k_1} = \mat{a+j_1c-k_1(b+j_1d)}{b+j_1d}{a+(j_1+1)c-k_1(b+(j_1+1)d)}{b+(j_1+1)d}.
\]
As $V_1$ does not have suffix $L$, we have
\[
(P_1V_1)_{2,1} < (P_1V_1)_{2,2} \implies
\]
\begin{equation} \label{eq:le_1}
\implies a+(j_1+1)c-k_1(b+(j_1+1)d) < b+(j_1+1)d.
\end{equation}
(The case $(P_1V_1)_{2,1} = (P_1V_1)_{2,2} \wedge (P_1V_1)_{1,1} = (P_1V_1)_{1,2}$ is not possible, because $\det(P_1V_1) = n >0$).

Further, we have
\[
P_2V_2 = LR^{j_2}ML^{-k_2} = \mat{a+j_2c-k_2(b+j_2d)}{b+j_2d}{a+(j_2+1)c-k_2(b+(j_2+1)d)}{b+(j_2+1)d} \in \D{}.
\]
Let now suppose for contradiction that $k_2 > k_1$. We have
\[
0 \leq (P_2V_2)_{2,1} = a+(j_2+1)c-k_2(b+(j_2+1)d) \leq a+(j_2+1)c-(k_1 + 1)(b+(j_2+1)d) =
\]
\[
= a- (k_1 + 1)b + (j_2+1)(c-(k_1 + 1)d) \overset{(k_1+1)d\, \geq \, d\, >\, c ; \;j_1<j_2}{<} a- (k_1 + 1)b + (j_1+1)(c-(k_1 + 1)d) =
\]
\[
= a + (j_1+1)c - (k_1 + 1)(b +(j_1+1)d)\overset{\eqref{eq:le_1}}{<}  0
\]
which is a contradiction.
\end{proof}

\begin{lemma} \label{le:2_hrany_do_P2}
Let $P_1,M \in \DB{}$, $\ell \geq 2$ and in the transducer $\T{}$ be the edge
\[
\Tedge{P_1}{L}{L^\ell }{M}.
\]
Then every other incoming edge of the state $M$ with the input word which has suffix $L$ is of the form
\[
\Tedge{P_2}{R^mL}{W}{M}
\]
for some $P_2 \in \DB{}, W \in \Dj$ and $m \geq 1$.
\end{lemma}
\begin{proof}
Let $\Tedge{P_2}{V}{W}{M}$ where $V \in \Dj$ and has suffix $L$, be some other incoming edge of the state $M$. By \Cref{le:stejna_pismena}, $W$ has prefix L. By \Cref{prop:assoc_sym}, we know that the edges $\Tedge{P_1}{L}{L^\ell }{M}$, $\Tedge{P_2}{V}{W}{M}$ exist if and only if the edges $\Tedge{M^T}{R^\ell }{R}{P_1^T}$ and $\Tedge{M^T}{W^T}{V^T}{P_2^T}$ exist.

Because $\ell  \geq 2$, we know according to \Cref{le:vstupni_slova} that $M^T$ has an outgoing edge with the input word $L$.
Further, by \Cref{thm:Raney_edge}, we obtain that every outgoing edge of the state $M^T$, which is not equal to $\Tedge{M^T}{R^\ell }{R}{P_1^T}$ and has input word with suffix $R$, has at least three runs.
Therefore, either $W^T = \widehat{W}LR^{j_1}$ or $W^T = \widehat{W} RL^{j_2}R$ for some $j_1 \geq 2,j_2 \geq 1$ and $\widehat{W} \in \Dj$. It means in the first case by the symmetric version of \Cref{le:vstup_RLk_1} of \Cref{le:vstup_V_1} and in the second case by the symmetric version of \Cref{le:vstup_LRjL_1} of \Cref{le:vstup_V_1} that $V^T = RL^m$ for some $m \geq 1$.

Therefore, $\Tedge{P_2}{V}{W}{M} = \Tedge{P_2}{R^mL}{W}{M}$ for some $m \geq 1$.
\end{proof}

\begin{lemma} \label{le:3_hrany_do_P2}
Let $\Tedge{P}{L^m}{W}{M}$, where $m \geq 1$, and $P= \mat{a}{b}{c}{d}, M = \mat{e}{f}{g}{h} \in \DB{}$. We have:
\begin{enumerate}
\item If $W = L^\ell$ for some $\ell \geq 1$, then $b=f$ and $a = e-mf$ and if $f \neq 0$, then $\frac{a}{b}< \frac{e}{f}$.

\item If $W = LR^\ell$ for some $\ell \geq 1$, then $b = f + \ell h > f \geq 0 $, $a = e + \ell g - m(f+\ell h) < e-mf$ and if $f \neq 0$, then $\frac{a}{b}< \frac{e}{f}$.
\end{enumerate}

\end{lemma}
\begin{proof}
In the first case, we obtain by direct computation that $b=f$, $a = e-mf$. Because $M \in \DB{}$, we have $f \geq 0$ and therefore $a \leq e$. Moreover, if $f \neq 0$, we have $\frac{a}{b} < \frac{e}{f}$.

We continue with the proof of the second case.
By direct computation, we obtain $b = f + \ell h$ and $a = e + \ell g - m(f+\ell h)$. Because $M \in \DB{}$, we have $h > f \geq 0$ and $g<h$.
So $e + \ell g - m(f+\ell h) = e-mf +\ell (g-h)m <e-mf$ and $f+\ell h > f \geq 0$. Therefore, $a<e$ and $b>f$, which for $f \neq 0$ means that $\frac{a}{b} <\frac{e}{f}$.
\end{proof}

\begin{lemma} \label{le:5_hrany_do_P2}
Let
\[
\Tedge{P_1}{L^{k_1}}{LR^{j_1}}{M} \text{ and } \Tedge{P_2}{L^{k_2}}{LR^{j_2}}{M},
\]
where $P_1 = \mat{e_1}{f_1}{g_1}{h_1}, P_2 = \mat{e_2}{f_2}{g_2}{h_2},M = \mat{a}{b}{c}{d} \in \DB{}$, $k_1,k_2 \geq 1,j_1,j_2 \geq 0$ and $j_2 >j_1$.
Then $e_2-(k_1-k_2)f_2<f_2$.
\end{lemma}
\begin{proof}
By direct computation, we obtain $e_2 = a + j_2c-k_2(b+j_2d),f_2 = b+j_2d, g_1 = a+(j_1+1)c-k_1(b+(j_1+1)d)$ and $h_1 = b+(j_1+1)d$. Since $P_1 \in \DB{}$, we have $g_1 <h_1$ and therefore the following series of inequalities holds.
\[
a + (j_1+1)c -k_1(b+(j_1+1)d) <b+(j_1+1)d
\]
\[
\implies
\]
\[
a  -(k_1+1)b <(j_1+1)[d(k_1+1) - c] \leq j_2[d(k_1+1) - c]
\]
(where the second inequality follows from $j_2 \geq j_1 +1$ and $[d(k_1+1) - c]>0$)
\[
\implies
\]
\[
a  + j_2c -k_2(b+j_2d)  < (b + j_2 d)(k_1-k_2+1)
\]
\[
\implies
\]
\[
e_2  < f_2(k_1-k_2+1),
\]
which is equivalent to the claim of the lemma.
\end{proof}

\begin{lemma} \label{le:6_hrany_do_P2}
For every $K = \mat{a}{b}{c}{d} \in \DB{}$ where $a>2b$, there exists an edge $ \Tedge{J}{L^t}{L^\ell }{K}$ where $t,\ell \geq 1$, and every other incoming edge of the state $K$ has $R$ in its output word.
\end{lemma}
\begin{proof}
We have $K^T \in \DB{}$ and by \Cref{thm:Raney_edge}, there exists an edge $\Tedge{K^T}{R^\ell }{W}{J^T}$ for some $\ell  \geq 1, W \in \Dj$ and $J \in \DB{}$ and there is no other outgoing edge of the state $K^T$ with input word, which does not contain $L$.
We have:
\[
K^TR^\ell  = \mat{a}{\ell a+c}{b}{\ell b+d}.
\]
Because $a>2b$, we have either $W = R$ or $W$ has prefix $R^2$. In the second case, we have by the symmetric version of \Cref{le:vstup_L_1} of \Cref{le:vstup_V_1} that $\ell  = 1$ and $W = R^t$ for some $t \geq 2$. The claim now follows from \Cref{prop:assoc_sym}.
\end{proof}

\begin{lemma} \label{le:7_hrany_do_P2}
Every state $J \in \DB{}$ has an incoming edge with the input word, which has suffix $L$.
\end{lemma}
\begin{proof}
We have $J^T \in \DB{}$. By \Cref{thm:Raney_edge}, there is always an edge $\Tedge{J^T}{V}{W}{N^T}$, which has an input word $V$ with suffix $R$. Moreover, by \Cref{le:stejna_pismena}, this edge has output word $W$, which has prefix $R$. It means that by \Cref{prop:assoc_sym}, there is also an edge $\Tedge{N}{W^T}{V^T}{J}$ where $W^T$ has suffix $L$.
\end{proof}

\begin{proposition} \label{prop:hrany_do_P2}
Let
\[
\Tedge{P_1}{L^{k_1}}{LR^{j_1}}{M} \quad \text{ and } \quad \Tedge{P_2}{L^{k_2}}{LR^{j_2}}{M},
\]
where $P_1, P_2,M \in \DB{}$, $k_1 > k_2 \geq 1$ and $j_2 >j_1 \geq 0$.
Then there is no walk of the form
\[
\Twalk{Q}{VL^{(k_1-k_2+1)}}{W}{P_2}
\]
in the transducer $\T{}$ and for arbitrary walk of the form
\[
\Twalk{Q}{VL^{(k_1-k_2)}}{W}{P_2}
\]
in the transducer $\T{}$ (if there is one) ($Q$ is the first state before reading the start of the run of $L$'s in the input word), we have
\[
\Twalk{Q}{VL^{(k_1-k_2)}}{W}{P_2} = \Tedge{Q}{R^pL}{W_1}{
\Twalk{Q_0}{L^{(k_1-k_2-1)}}{W_2}{P_2}
},
\]
where $W_1W_2 = W$, $Q_0 \in \DB{}$ and $p \geq 1$.
\end{proposition}
\begin{proof}
By \Cref{le:3_hrany_do_P2}, and since $j_2 > 0$, we have $(P_2)_{1,2} \neq 0$.
An arbitrary walk of the form
\[
\Twalk{Q}{VL^{r}}{W}{P_2}
\]
in the transducer $\T{}$, where $r \geq 1$ and $Q$ is the first state before reading the start of the run of $L$'s in the input word, can be decomposed in the following way.
\[
\Twalk{Q}{VL^{r}}{W}{P_2} = \Tedge{Q}{VL^{u_0}}{W_0}{
\Tedge{Q_0}{L^{u_1}}{W_1}{Q_1} \cdots \Tedge{Q_{m-1}}{L^{u_m}}{W_m}{Q_m}
}
\]
where $Q_m = P_2$, $m \in \N$ and for all $i \in \N, i \leq m$,  $W_i \in \Dj, Q_i \in \DB{}$ and $u_i \geq 1$.

By \Cref{le:3_hrany_do_P2}, we have $(Q_i)_{1,2} \neq 0$ for all $i \in \N, i \leq m$ and $\frac{(Q_i)_{1,1}}{(Q_i)_{1,2}} < \frac{(Q_{i+1})_{1,1}}{(Q_{i+1})_{1,2}} $ for all $i \in \N, i \leq {m-1}$.

Let $r$ be maximal possible such that there exists the walk of the form $\Twalk{Q}{VL^{r}}{W}{P_2}$.
\Cref{le:7_hrany_do_P2} shows that there is an incoming edge of the state $Q$, which has suffix of its input word equal to $L$.
Therefore, $V$ has suffix $R$.
We know that $u_0 \geq 1$ and therefore by \Cref{le:stejna_pismena}, $W_0$ has prefix $L$. Now $u_0 = 1$ holds or by \Cref{le:vstup_RLk_1} of \Cref{le:vstup_V_1}, we have $W_0 = LR^s$ for some $s \geq 1$.

We prove by contradiction that $(Q_0)_{1,1} \leq 2 (Q_0)_{1,2}$.
We suppose that $(Q_0)_{1,1} > 2 (Q_0)_{1,2}$. It follows from \Cref{le:6_hrany_do_P2} that there exists some edge  $\Tedge{J}{L^t}{L^{\ell}}{Q_0}$ where $t, \ell \geq 1$. By \Cref{le:vstup_L_1} of \Cref{le:vstup_V_1} either $\ell = 1$ or $t=1$.
By \Cref{le:1_hrany_do_P2} or by \Cref{le:2_hrany_do_P2} or because $t \geq 1$, we have $t \geq u_0$.
By \Cref{le:7_hrany_do_P2}, there is an edge, which ends in the state $J$ and has input word with suffix $L$.
This is a contradiction with the maximality of $r$.

Therefore, we have $(Q_0)_{1,1} \leq 2 (Q_0)_{1,2}$, which by \Cref{le:vstup_do_N} means that $u_0 =1$ and $V = R^p$ for some $p \geq 1$.

Let $i \geq 1$, $i \leq m$. Further, by \Cref{le:vstup_L_1} of \Cref{le:vstup_V_1} and \Cref{le:3_hrany_do_P2}, we have
\[
(Q_{i-1})_{1,1} \leq (Q_i)_{1,1} - u_i (Q_i)_{1,2}
\]
and
\[
(Q_{i-1})_{1,2} \geq (Q_i)_{1,2}.
\]

Because $Q_0 \in \DB{}$, we have $(Q_0)_{1,1} >  (Q_0)_{1,2}$. It means that
\[
(Q_m)_{1,1} - (r-1) (Q_m)_{1,2} = (Q_m)_{1,1} - (u_1 + u_2 + \dots u_m) (Q_m)_{1,2} \geq(Q_0)_{1,1} >  (Q_0)_{1,2} \geq  (Q_m)_{1,2}.
\]

At the same time, we have according to \Cref{le:5_hrany_do_P2} that
\[
(Q_m)_{1,1}-(k_1-k_2)(Q_m)_{1,2}<(Q_m)_{1,2}.
\]
Because $(Q_m)_{1,2} \geq 1$, we have $r-1 < k_1 - k_2$ and the proposition holds.
\end{proof}

The symmetry of $\T{}$ can be used for the following observation about the output words.

\begin{lemma} \label{le:vystupy_ruzne}
Let $\Tedge{M_1}{V_1}{W_1}{P}$ and $\Tedge{M_2}{V_2}{W_2}{P}$ for some $M_1,M_2,P \in \DB{}$, $V_1,V_2,W_1,W_2 \in \Dj$ be two different edges in the transducer $\T{}$.
The word $W_1$ is not a suffix of $W_2$ or vice versa.
\end{lemma}
\begin{proof}
By \Cref{prop:assoc_sym}, the edges $\Tedge{P^T}{W_1^T}{V_1^T}{M_1^T}$ and $\Tedge{P^T}{W_2^T}{V_2^T}{M_2^T}$ exist. The claim follows from \Cref{thm:Raney_edge}.
\end{proof}

We may now proceed with a proof of \Cref{co:vyfukovani_ze_symetrickeho}.
First, we recall its statement:

\restatableVyfukovani*

\begin{proof}[Proof of \Cref{co:vyfukovani_ze_symetrickeho}]
Let $(p_i)_{i=1}^{g}$ be the sequence of all the transitions taken on the walk $\Twalk{M}{V}{W}{M}$, ordered as they appear on this walk.
We shall transform each of the transition $p_i$ into a new walk $\widehat{p}_i$ having the same starting state and ending state as $p_i$.
Doing that, we shall produce a new walk from $M$ to $M$ given by the sequence $(\widehat{p}_i)_{i=1}^{g}$.

Let $U_iQ_i^{j_i}$ be the input word of the transition $p_i$ with $Q_i \in \{L,R\}$, $j_i > 0$, and $U_i$ empty or ending in a letter distinct from $Q_i$.
If $U_i$ is not empty, $F_i$ denotes the first letter of $U_i$.

We proceed from $i = g$ to $i=1$ and replace $p_i$ with $\widehat{p}_i$ using the following rules. In the case $i = 1$, we define $i-1 = g$.
The rules use \Cref{coro:kappa_hrana} or its symmetric version and construct walks $\widehat{p}_i$ from the walks given by an appropriate item of \Cref{coro:kappa_hrana}:
\begin{enumerate}[I.]
  \item if $U_i$ is empty and $Q_i \neq Q_{i-1}$, then apply \cref{it:kappa_hrana_prazdneV} of \Cref{coro:kappa_hrana} if $Q_i = R$ or its symmetric version otherwise; \label{it:nafuk_krok_prazdne}
  \item if $U_i$ is empty and $Q_i = Q_{i-1}$, then $\widehat{p}_i = p_i$; \label{it:nafuk_krok_kopie}
  \item if $U_i$ is not empty and $F_i = Q_{i-1}$, then apply \cref{it:kappa_hrana_nic} of \Cref{coro:kappa_hrana} if $Q_i = R$ or its symmetric version otherwise;
  \item if $U_i$ is not empty and $F_i \neq Q_{i-1}$, then
  \begin{enumerate}[label=\Roman{enumi}\alph*.]
    \item if $Q_{i-1} = L$ and $Q_i = R$, apply \cref{it:kappa_hrana_L} of \Cref{coro:kappa_hrana};
    \item if $Q_{i-1} = R$ and $Q_i = R$, apply \cref{it:kappa_hrana_R} of \Cref{coro:kappa_hrana};
    \item if $Q_{i-1} = R$ and $Q_i = L$, apply \cref{it:kappa_hrana_L} of symmetric version of \Cref{coro:kappa_hrana};
    \item if $Q_{i-1} = L$ and $Q_i = L$, apply \cref{it:kappa_hrana_R} of symmetric version of \Cref{coro:kappa_hrana}.
  \end{enumerate}
\end{enumerate}

Let $\Twalk{M}{U}{W_U}{M}$ be the walk composed of the walks $\widehat{p}_1,\widehat{p}_2,\ldots,\widehat{p}_g$.
It follows from the above construction that all the runs in $U$, except for the last run, are of length at least $4n$ and $Q_g^{4n}$ is a prefix of $U$.
As the new walks $\widehat{p}_i$ are given by \Cref{coro:kappa_hrana} or kept the same, the number of runs in the output words is either always strictly increased or the output word is the same.
Since $\sigma(\widehat{A_1}) > \sigma(A_1)$ and $\sigma(\widehat{A_2}) > \sigma(A_2)$ imply $\sigma(\widehat{A_1}\widehat{A_2}) > \sigma(A_1A_2)$ for all $A_1, A_2, \widehat{A_1}, \widehat{A_2} \in \Dj$, we conclude that using the above rules, the number $\sigma(W_U)$ may not be less than $\sigma(W)$.
We conclude that $\sigmaC(W_U) \geq \sigmaC(W)$.

We shall now repeatedly apply \Cref{le:vyfouknuti_n} to the walk $\Twalk{M}{U}{W_U}{M}$ to decrease the length of most of the runs in the input word between $4n$ and $5n-1$, without changing the output words except for decreasing lengths of some of their runs.
We end up with a walk
\[
\Twalk{M}{Q_g^{e}U'Q_g^{f}}{W_U'}{M}
\]
where $\sigma(W_U) = \sigma(W_U')$, $U' = \kappa(U')$ and $U'$ starts and ends with the letter distinct from $Q_g$.
It remains to deal with the first and the last run, which are both runs of the same letter $Q_g$.
In order to do that, we shift the start and the end of the closed walk $\Twalk{M}{Q_g^{e}U'Q_g^{f}}{W_U'}{M}$ to another state, denoted by $\widehat{M}$, on this closed walk such that the run $Q_g^{e+f}$ is inside the input word.
This is possible due to the fact that $U'$ contains a run of length at least $4n$ of the letter distinct from $Q$.
Let $W'$ be the output word of this shifted closed walk.
By the definition, we have $\sigmaC(W') = \sigmaC(W_U')$.
We now apply \Cref{le:vyfouknuti_n} one last time to reduce the length of the run $Q_g^{e+f}$ in the input word of the shifted closed walk.
We obtain a new output word $\widehat{W}$ which satisfies
\[
\sigmaC(\widehat{W}) = \sigmaC(W') = \sigmaC(W_U') = \sigmaC(W_U) \geq \sigmaC(W).
\]
As the input word of this closed walk belongs to $\tau_\kappa(V)$, the first part of the proof is finished.

Now, we prove \Cref{it:1)vyfukovani,it:2)vyfukovani}.
The proofs of the two claims of \Cref{it:1)vyfukovani,it:2)vyfukovani} are very similar and therefore we prove them together.
In what follows, $\C$ denotes the closed walk $\Twalk{M}{V}{W}{M}$ and $\widehat{\C}$ the closed walk $\Twalk{\widehat{M}}{\widehat{V}}{\widehat{W}}{\widehat{M}}$.

The step \ref{it:nafuk_krok_kopie} of the algorithm in the first part of the proof may not be applied to all of the edges $p_i$.
In other words, there is an edge $p_i$ such that \Cref{coro:kappa_hrana} is applied to it in the algorithm.
Moreover, we may assume that such edge satisfies $p_i = \Tedge{S}{V_3R^j}{\bullet}{T}$ where $V_3$ does not have suffix $R$.
It follows from \Cref{le:max_n_jednoho_pismene} that $j \leq n$ and if $V_3$ is empty, then we are in the step \ref{it:nafuk_krok_prazdne}

Let $X \in \{\varepsilon, L^{4n}, R^{4n}\}$.
The application of \Cref{coro:kappa_hrana} to $p_i$ produces a walk $\widehat{p}_i = \Twalk{S}{X\kappa(V_3)R^j}{\bullet}{T}$ on the closed walk $\Twalk{\widehat{M}}{\widehat{V}}{\widehat{W}}{\widehat{M}}$.
If we change the starting state of this closed walk to $S$, we obtain some walk $\Twalk{S}{B}{\bullet}{S}$.
In the first case (the case \ref{it:1)vyfukovani} of \Cref{co:vyfukovani_ze_symetrickeho}), we have $\Twalk{S}{B}{\bullet}{S} = (\Twalk{S}{B_1}{\bullet}{S})^{m_1}$ for $B_1^{m_1} = B$ so there is a walk $\widehat{p}_q$, $i \neq q$ and $\widehat{p}_i = \widehat{p}_q$ such that if we change the starting vertex of the closed walk $\Twalk{S}{B}{\bullet}{S}$ to the starting vertex of $\widehat{p}_q$, we obtain again the closed walk $\Twalk{S}{B}{\bullet}{S}$. In the second case (the case \ref{it:2)vyfukovani} of \Cref{co:vyfukovani_ze_symetrickeho}), the symmetricity of the walk $\Twalk{S}{B}{\bullet}{S}$ implies that we may also change the starting vertex to $\assoc{S}$, and the closed walk is $\Twalk{\assoc{S}}{\assoc{B}}{\bullet}{\assoc{S}}$ where the first walk is $\widehat{p}_q = \assoc{\widehat{p}_i} = \Twalk{\assoc{S}}{\assoc{X\kappa(V_3)}L^j}{\bullet}{\assoc{T}}$. Because we are investigating both cases together, we put $T_0 = T$ for the first case and $T_0 = \assoc{T}$ for the second case.

We shall now investigate the edge $p_{\ell}$ on the original closed walk, which, after application of the algorithm in the first part of this proof, produced the start of the reading of the last run in the input word of $\widehat{p}_q$. This last run is denoted $E^j$, where $E = R$ for the first case and $E = L$ for the second case.

We are again sure that \Cref{coro:kappa_hrana} is applied to $p_{\ell}$ since we are tracking a start of a run.

\begin{enumerate}[A)]
\item \label{it:A)vyfukovani}
 $p_\ell = \Tedge{S_1}{ZE^k}{\bullet}{T_1}$ and $\widehat{p_\ell} = \Twalk{S_1}{X'\kappa(Z)E^k}{\bullet}{T_1}$ with $k$ maximal possible, $Z \in \Dj$ and $X' \in \{\varepsilon, L^{4n}, R^{4n}\}$.

If $k=j$, then $T_1 =T_0$. Since for the first case $V = (V_1)^{m_1}$, we arrive in the closed walk $\C$ at the state $T_0= T$ at least two times with the same input and therefore $\C = (\C_1)^{m_2}$ for some $m_2 \geq 2$, which means that $\Twalk{M}{V}{W}{M}=(\Twalk{M}{V_2}{W_2}{M})^{m_2}$.
In the second case, we have a similar situation.
Since $V = V_1\assoc{V_1}$, i.e., the original input word is symmetric itself, we arrive in the closed walk ${\C}$ at the state $\assoc{T}$ with the input word, which is symmetric to the input word after the edge $p_i$, and so $\Twalk{M}{V}{W}{M}$ is symmetric.

If $k < j$, then the walk $\widehat{p}_q$ ends with the walk $\Twalk{T_1}{E^{j-k}}{\bullet}{T_0}$.
The walk $\Twalk{T_1}{E^{j-k}}{\bullet}{T_0}$ is taken after $p_\ell$.
Thus, we arrive in the closed walk $\C$ at the state $T_0$ with the input word, which is either the same (in the first case) or symmetric (in the second case) to the input word after the walk $p_i$, and so either $\Twalk{M}{V}{W}{M}=(\Twalk{M}{V_2}{W_2}{M})^{m_2}$ for some $m_2 \geq 2$ or $\Twalk{M}{V}{W}{M}$ is symmetric, respectively.

If $k > j$, then the walk $\widehat{p}_\ell$ ends with the walk $\Twalk{T_0}{E^{k-j}}{\bullet}{T_1}$.
Thus, the walk $\Twalk{T}{R^{k-j}}{\bullet}{T_2}$, where $T_2 = T_1$ in the first case and $T_2 = \assoc{T_1} $ in the second case, exists.
Therefore, as the observed run of $R$'s in $p_i$ is of length at least $k$, it is followed by $\Twalk{T}{R^{k-j}}{\bullet}{T_2}$.
Again, we either arrive in the closed walk $\C$ two times at the state $T_1$ with the same input word or we find two states, $\assoc{T_1}$ and $T_1$ that have symmetric input words, and so either $\Twalk{M}{V}{W}{M}=(\Twalk{M}{V_2}{W_2}{M})^{m_2}$ or $\Twalk{M}{V}{W}{M}$ is symmetric, respectively.

\begin{figure}
\centering
\begin{tikzpicture}[small/.style={fontscale=-1}]
  \matrix (m) [matrix of math nodes,row sep=3em,column sep=4.4em,minimum width=2em]
  {
     S_1  & \phantom{N}  &  \phantom{N} &  \phantom{N} & \phantom{N} & \phantom{N}  & T_1 \\
     \phantom{D}   &  D_{f} & D_{g+1} & D_g & D_{g-1}  &  D_0 &  \\};

\path (m-1-6.center) edge[->] node[above]{$Y^{r_1}$} (m-1-7);
\draw[] (m-1-5.center) [-] -- ($(m-1-5)!0.25!(m-1-6)$) ($(m-1-5)!0.25!(m-1-6)$) edge[dotted] ($(m-1-5)!0.75!(m-1-6)$) ($(m-1-5)!0.75!(m-1-6)$) edge[-] node[above]{$\assoc{Y}^{i_0}$} (m-1-6.center);
\path (m-1-4.center) edge[-] node[above]{$\assoc{E}^{i_{g-1}}$} (m-1-5.center);
\path (m-1-3.center) edge[-] node[above]{$E^k$} (m-1-4.center);
\draw[] (m-1-2.center) [-] -- ($(m-1-2)!0.25!(m-1-3)$) ($(m-1-2)!0.25!(m-1-3)$) edge[dotted] ($(m-1-2)!0.75!(m-1-3)$) ($(m-1-2)!0.75!(m-1-3)$) edge[-] (m-1-3.center);
\path (m-1-1) edge[-] node[above]{$P^{i_{f}}$} (m-1-2.center);

\path (m-1-6.center) edge[->] node[left]{$\assoc{Y}^{t_0}$} (m-2-6);
\path (m-1-5.center) edge[->] node[left]{$\assoc{E}^{t_{g-1}}$} (m-2-5);
\path (m-1-4.center) edge[->] node[left]{$E^{t_g}$} (m-2-4);
\path (m-1-3.center) edge[->] node[left]{$\assoc{E}^{t_{g+1}}$} (m-2-3);

\path (m-1-2.center) edge[->] node[left]{$P^{t_f}$} (m-2-2);

\path (m-2-6) edge[->>,out=0,in=-90] node[below,yshift=-0.7em]{$\assoc{Y}^{4n-t_0}Y^{r_1}$} (m-1-7);
\draw[] (m-2-5) edge[-] node[above,small,xshift=0.5em]{ } ($(m-2-5)!0.25!(m-2-6)$) ($(m-2-5)!0.25!(m-2-6)$) edge [dotted] ($(m-2-5)!0.75!(m-2-6)$) ($(m-2-5)!0.75!(m-2-6)$) edge[->>] node[above,small]{$Y^{4n-t_1}\assoc{Y}^{i_0+t_0}$}  (m-2-6);
\path (m-2-4) edge[->>] node[above,small]{$E^{4n-t_g}\assoc{E}^{i_{g-1}+t_{g-1}}$} (m-2-5);
\path (m-2-3) edge[->>] node[above,small]{$\assoc{E}^{4n-t_{g+1}}E^{k+t_{g}}$} (m-2-4);
\draw[] (m-2-2) edge[-] node[above,small,xshift=0.5em]{ } ($(m-2-2)!0.25!(m-2-3)$) ($(m-2-2)!0.25!(m-2-3)$) edge [dotted] ($(m-2-2)!0.75!(m-2-3)$) ($(m-2-2)!0.75!(m-2-3)$) edge[->>]  (m-2-3);


\node at (m-2-1){$D_{f+1}$};
\path (m-1-1) edge[->,dashed] node[right]{$\assoc{P}^{t_{f+1}}$} (m-2-1);
\path (m-2-1) edge[->>,dashed] node[above,small]{$\assoc{P}^{4n-t_{f+1}}P^{i_f+t_f}$} (m-2-2);

\draw[rounded corners,dashed] ([yshift=0.5em]m-1-1.south east) -- ([xshift=0.7em,yshift=0.5em]m-1-2.south west) [->] --  ([xshift=0.7em]m-2-2.north west);


\filldraw[fill=white] (m-1-6) circle (2pt);
\filldraw[fill=white] (m-1-5) circle (2pt);
\filldraw[fill=white] (m-1-4) circle (2pt);
\filldraw[fill=white] (m-1-3) circle (2pt);
\filldraw[fill=white] (m-1-2) circle (2pt);


\node [label={[label distance=-0.2cm]-90:$T_0$}] (pf1) at ($(m-2-3)!0.7!(m-2-4)$) {};
\draw[mark=triangle*,mark size=2.5pt,mark options={color=red}] plot coordinates {(pf1)};


\end{tikzpicture}
\caption{
The situation of \Cref{it:co:vyfukovani_ze_symetrickeho_spor} in the proof of \Cref{co:vyfukovani_ze_symetrickeho}, based on \Cref{fig:nafukovani}.
On the top line we have the edge $p_\ell$, which is transformed using the procedure given by \Cref{coro:kappa_hrana}.
We find an intermediate state $D_g$ connected to the run $E^k$.
We identify the state $T_0$ on the bottom line, the new walk $\widehat{p}_\ell$, marked by a triangle.
}
\label{fig:nafukovani_vyfukovani_dukaz}
\end{figure}
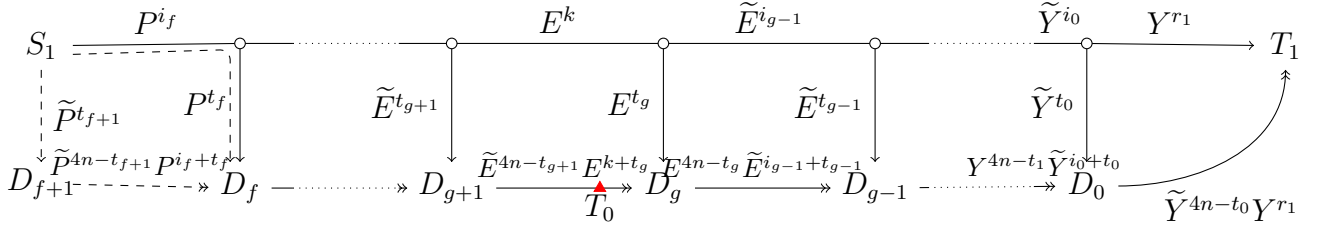

\item \label{it:co:vyfukovani_ze_symetrickeho_spor}

$p_\ell = \Tedge{S_1}{Z_1E^kZ_2Y^{r_1}}{\bullet}{T_1}$, with $Z_1,Z_2 \in \Dj$, $Y \in \left\{ L,R \right\}, r_1 \geq 1$ and $k\geq 1$ maximal possible.
As $V = (V_1)^{m_1}$ or $V = V_1\assoc{V_1}$, we have $k \geq j$.

We have $\widehat{p}_\ell = \Twalk{S_1}{X'\kappa(Z_1E^kZ_2)Y^{r_1}}{\bullet}{T_1}$ where $X' \in \{\varepsilon, L^{4n}, R^{4n}\}$.
As $\Twalk{\widehat{M}}{\widehat{V}}{\widehat{W}}{\widehat{M}} = (\Twalk{\widehat{M}}{\widehat{V_1}}{\widehat{W_1}}{\widehat{M}})^{m_1}$ or is symmetric, we know that after reading $E^j$ in the run $\kappa(E^k)$ we arrive at $T_0$.
The situation is illustrated in \Cref{fig:nafukovani_vyfukovani_dukaz}.
We arrive at the intermediate state $D_g$ by reading $E^{k+t_g}$, where $t_g \geq 1$, which is also read when taking the vertical path to $D_g$. Therefore, $\Tedge{S_1}{Z_1E^{k+t_g}}{W_3}{D_g}$ and $\Twalk{S_1}{Z_1\assoc{E}^{4n}E^{k+t_g}}{W_4}{D_g} = \Twalk{S_1}{Z_1\assoc{E}^{4n}E^j}{\bullet}{ \Twalk{T_0}{E^{k+t_g - j}}{\bullet}{D_g}}$ for some $W_3,W_4 \in \Dj$ and we can find a state $T_2 \in \DB{}$ such that $\Twalk{S_1}{Z_1\assoc{E}^{4n}E^{k+t_g}}{W_4}{D_g} = \Twalk{S_1}{Z_1\assoc{E}^{4n}E^j}{W_5}{\Twalk{T_0}{E^{k+t_g -r - j}}{W_6}{ \Tedge{T_2}{E^r}{W_7}{D_g}}}$ for some $W_5, W_6,W_7\in \Dj$ and $k+t_g-j\geq r \geq 1$.

By \Cref{le:nafouknuti_jedno} or its symmetric version, the word $W_4$ has suffix $E^{-1}W_3$ and because $k+t_g \geq 2$, we have by \Cref{le:vstup_V_1} \Cref{le:vstup_RLk_1} for $Z_1$ nonempty or by \Cref{le:vstup_V_1} \Cref{le:vstup_L_1} otherwise that $W_3 = E\assoc{E}^z$, where $z \geq 0$.
And for $Z_1$ nonempty even $z \geq k+t_g - 1 \geq 1$. So $W_4$ has suffix $\assoc{E^z}$.
Moreover, by \Cref{le:vstup_V_1} \Cref{le:vstup_L_1} or its symmetric version, we have $W_7 = E\assoc{E}^s$ for some $s \geq 0$ or $W_7 = E^t$ for some $t \geq 2$.
If $W_7 = E^t$ for some $t \geq 2$, then $r = 1$ and by \Cref{le:2_hrany_do_P2}, there cannot be the edge $\Tedge{S_1}{Z_1E^{k+t_g}}{W_3}{D_g}$, which is a contradiction. Therefore, $W_7 = E\assoc{E}^s$ for some $s \geq 0$.
Further, we know that $W_7$ is a suffix of $W_4$ and therefore $\assoc{E}^z$ is a suffix of $W_7$. Moreover, by \Cref{le:vystupy_ruzne}, $W_3 \neq W_7$, which means that $s>z$.

We distinguish the two following cases:

\begin{enumerate}[1)]
\item \label{it:B)vyfukovani}
$Z_1$ is not the empty word.

In this case, we have $s \geq z+1 \geq k + t_g \geq k + t_g -j + 1\geq r+1$.
Let $T_2 = \mat{a}{b}{c}{d}$. Using \Cref{le:hrany_do_P} or its symmetric version on the edge $\Tedge{T_2}{E^{r}}{E\assoc{E}^s}{D_g}$, we have $a <2b$ for $E = L$ and $a<2c$ for $E=R$.

Moreover, let $\Tedge{N}{V_8}{W_8}{T_2}$, where $N \in \DB{}, V_8,W_8 \in \Dj$ and $V_8$ has suffix $E$, be an edge in the transducer $\T{}$. By \Cref{le:vstup_do_N} or its symmetric version, we have $V_8 = \assoc{E}^yE$ for some $y \geq 1$.

Therefore, the walk $\Twalk{T_0}{E^{k+t_g -r - j}}{W_6}{T_2}$ is empty and $T_0 = T_2$. Therefore also for every edge $\Tedge{N}{V_8}{W_8}{T_0}$, where $V_8$ has suffix $E$, we have $V_8 = \assoc{E}^yE$. Specially we have $p_i = \Tedge{S}{L^yR}{\bullet}{T}$ for some $y \geq 1$ (in the second case, we have used \Cref{prop:assoc_sym}) which means that $j=1$.

Because $Z_1$ is not empty, we can write $p_\ell = \Tedge{S_1}{Z_3\assoc{E^{i_{g+1}}}E^kZ_2Y^{r_1}}{\bullet}{T_1}$, with $i_{g+1} \geq 1$ maximal possible and $Z_3 \in \Dj$.

Now we can have one of the following situations.

\begin{enumerate}[i)]
\item \label{it:a)vyfukovani}
 The edge $p_{i-1}$ has input word with a suffix $L$ and $Z_3$ is not an empty word.

We find the edge $p_{u}$ on the original closed walk on which starts the reading of the run $L^y$ of the input word of $p_i$ and we apply \Cref{it:co:vyfukovani_ze_symetrickeho_spor} \ref{it:B)vyfukovani} on the edges $p_u$ and $p_{\ell}$.

\item \label{it:b)vyfukovani}
The edge $p_{i-1}$ has input word with a suffix $L$, $Z_3$ is empty and the edge $p_{\ell - 1}$ has input word with a suffix $\assoc{E}$.

In this case, we find the edge $p_u$ on the original closed walk on which starts the reading of the run $L^y$ of the input word of $p_i$ and the edge $p_v$ on the original closed walk on which starts the reading of the run $\assoc{E}^{i_{g+1}}$ of the input word of $p_{\ell}$  and we apply \Cref{it:A)vyfukovani} on these two edges.

\item \label{it:c)vyfukovani}
 The edge $p_{i-1}$ has input word with a suffix $L$, $Z_3$ is empty and the edge $p_{\ell - 1}$ has input word with a suffix $E$.

In this case, we find the edge $p_u$ on the original closed walk on which starts the reading of the run $L^y$ of the input word of $p_i$ and we apply \Cref{it:co:vyfukovani_ze_symetrickeho_spor} \ref{it:C)vyfukovani} on the edges $p_u$ and $p_{\ell}$.

\item \label{it:d)vyfukovani}
The edge $p_{i-1}$ has input word with a suffix $R$ and $Z_3$ has at least two runs.

We find the edge $p_u$ on the original closed walk on which starts the reading of the run of $R$'s, which ends as a suffix of the input word of the edge $p_{i-1}$ and we apply \Cref{it:co:vyfukovani_ze_symetrickeho_spor} \ref{it:B)vyfukovani} on the edges $p_u$ and $p_{\ell}$.

\item \label{it:e)vyfukovani}
The edge $p_{i-1}$ has input word with a suffix $R$, $Z_3 = E^{i_{g+2}}$ and the edge $p_{\ell - 1}$ has input word with a suffix $E$.

We find the edge $p_u$ on the original closed walk on which starts the reading of the run of $R$'s, which ends as a suffix of the input word of the edge $p_{i-1}$ and the edge $p_v$ on the original closed walk on which starts the reading of the run $E^{i_{g+2}}$ of the input word of $p_{\ell}$ and we apply \Cref{it:A)vyfukovani} on the edges $p_u$ and $p_v$.

\item \label{it:f)vyfukovani}
The edge $p_{i-1}$ has input word with a suffix $R$, $Z_3 = E^{i_{g+2}}$ and the edge $p_{\ell - 1}$ has input word with a suffix $\assoc{E}$.

We find the edge $p_u$ on the original closed walk on which starts the reading of the run of $R$'s, which ends as a suffix of the input word of the edge $p_{i-1}$ and we apply \Cref{it:co:vyfukovani_ze_symetrickeho_spor} \ref{it:C)vyfukovani} on the edges $p_u$ and $p_{\ell}$.

\item \label{it:g)vyfukovani}
The edge $p_{i-1}$ has input word with a suffix $R, Z_3$ is empty and the edge $p_{\ell - 1}$ has input word with a suffix $\assoc{E}$.

We find the edge $p_v$ on the original closed walk on which starts the reading of the run $\assoc{E}^{i_{g+1}}$ of the edge $p_\ell$ and we apply \Cref{it:co:vyfukovani_ze_symetrickeho_spor} \ref{it:C)vyfukovani} on the edges $p_v$ and $p_i$.

\item \label{it:h)vyfukovani}
The edge $p_{i-1}$ has input word with a suffix $R$, $Z_3$ is empty and the edge $p_{\ell - 1}$ has input word with a suffix $E$.

We find the edge $p_u$ on the original closed walk on which starts the reading of the run of $R$'s, which ends as a suffix of the input word of the edge $p_{i-1}$ and the edge $p_v$ on the original closed walk on which starts the reading of the run $E$, which ends as a suffix of the input word of the edge $p_{\ell - 1}$ and we apply \Cref{it:A)vyfukovani} on the edges $p_u$ and $p_v$.

\end{enumerate}

Since the word $Z_3$ is finite, we are sure that after a finite number of applications of \Cref{it:B)vyfukovani}, one of the possibilities \ref{it:b)vyfukovani},\ref{it:c)vyfukovani},\ref{it:e)vyfukovani},\ref{it:f)vyfukovani},\ref{it:g)vyfukovani} or \ref{it:h)vyfukovani} occurs.

\item \label{it:C)vyfukovani}
$Z_1$ is the empty word.

In this case, we have $\Tedge{S_1}{Z_1E^{k+t_g}}{W_3}{D_g} = \Tedge{S_1}{E^{k+t_g}}{E\assoc{E}^z}{D_g}$ and $\Tedge{T_2}{E^{r}}{W_7}{D_g} = \Tedge{T_2}{E^{r}}{E\assoc{E}^s}{D_g}$, where $s>z$.
Using \Cref{prop:hrany_do_P2} or its symmetric version on the edges $\Tedge{S_1}{E^{k+t_g}}{E\assoc{E}^z}{D_g}$ and
$\Tedge{T_2}{E^r}{E\assoc{E}^s}{D_g}$, we get that for every walk of the form $\Twalk{P}{V_8E^{k+t_g-r}}{W_8}{T_2}$ for some $P \in \DB{}$ and $V_8,W_8 \in \Dj$ we have $\Twalk{P}{V_8E^{k+t_g-r}}{W_8}{T_2} = \Tedge{P}{\assoc{E}^vE}{W_9}{\Twalk{N_1}{E^{k+t_g-r-1}}{W_{10}}{T_2}}$ for some $N_1 \in \DB{}$ and $v \geq 1$.
It holds specially for every such walk which end with the walk $\Twalk{T_0}{E^{k+t_g -r - j}}{W_6}{T_2}$ and therefore in the case \Cref{it:1)vyfukovani} we have ($T = T_0$, $E = R$):
\[
 \Tedge{S}{V_3R^j}{\bullet}{ \Twalk{T}{R^{k+t_g -r - j}}{W_6}{T_2}}=   \Tedge{S}{L^vR}{W_9}{ \Twalk{T}{R^{k+t_g -r - 1}}{W_6}{T_2}}
\]
 and in the case \Cref{it:2)vyfukovani}, we have ($T_0 = \assoc{T}, E = L$):
\[
 \Tedge{\assoc{S}}{\assoc{V_3}L^j}{\bullet}{  \Twalk{\assoc{T}}{L^{k+t_g -r - j}}{W_6}{T_2}} = \Tedge{\assoc{S}}{R^vL}{W_9}{ \Twalk{\assoc{T}}{L^{k+t_g-r-1}}{W_{10}}{T_2}}.
\]
Therefore, (in the second case by \Cref{prop:assoc_sym}) we have $j = 1$ and $p_i = \Tedge{S}{L^vR}{\bullet}{T}$ for some $v \geq 1$.

Now we can have two situations.

\begin{enumerate}[i)]
\item The edge $p_{i-1}$ has  $L$ as a suffix of its input word.

We find the edge $p_{u_1}$ on the original closed walk on which starts the reading of the run of $L$'s, which ends on the edge $p_{i}$ and the edge $p_{u_2}$ on the original closed walk on which starts the reading of the run $\assoc{E}$, which ends as a suffix of the input word of the edge $p_{\ell - 1}$ and we apply \Cref{it:A)vyfukovani} on the edges $p_{u_1}$ and $p_{u_2}$.

\item The edge $p_{i-1}$ has  $R$ as a suffix of its input word.

We can find the edge $p_{u_2}$ on the original closed walk on which starts the reading of the run $\assoc{E}$, which ends as a suffix of the input word of the edge $p_{\ell - 1}$ and we apply \Cref{it:co:vyfukovani_ze_symetrickeho_spor} \ref{it:C)vyfukovani} on the edges $p_{u_2}$ and $p_i$.

\end{enumerate}
Now we can see that there either exist two edges $p_{u_1}, p_{u_2}$ on which we can apply \Cref{it:A)vyfukovani} or that for all $w$ we have $p_{i_w}= \Tedge{M_{i_w}}{L^{q_{i_w}}R}{\bullet}{N_{i_w}}$ or $p_{i_w}= \Tedge{M_{i_w}}{R^{q_{i_w}}}{\bullet}{N_{i_w}}$  and $p_{\ell_w}= \Tedge{M_{\ell_w}}{{E}^{q_{\ell_w}}\assoc{E}}{\bullet}{N_{\ell_w}}$ or $p_{\ell_w}= \Tedge{M_{\ell_w}}{\assoc{E}^{q_{\ell_w}}}{\bullet}{N_{\ell_w}}$, where $q_{i_w},q_{\ell_w} \geq 1$ and $ i_w \equiv i-w \mod g, \ell_w \equiv \ell-w \mod g$ where $g$ is such that $\C = (p_i)_{i= 1}^{g}$.

For the final part of the proof, we split the cases according to \Cref{it:1)vyfukovani,it:2)vyfukovani} of the statement.
\begin{enumerate}[(a)]
\item
$p_{i_w}= \Tedge{M_{i_w}}{L^{q_{i_w}}R}{\bullet}{N_{i_w}}$ or $p_{i_w}= \Tedge{M_{i_w}}{R^{q_{i_w}}}{\bullet}{N_{i_w}}$ and $p_{\ell_w}= \Tedge{M_{\ell_w}}{{R}^{q_{\ell_w}}L}{\bullet}{N_{\ell_w}}$ or $p_{\ell_w}= \Tedge{M_{\ell_w}}{L^{q_{\ell_w}}}{\bullet}{N_{\ell_w}}$, for all $w$ and for $q_{i_w}, q_{\ell_w} \geq 1$, which is a contradiction with $\C = (p_{i_w})_{w= 1}^{g} = (p_{\ell_w})_{w= 1}^{g}$.
\item
$p_{i_w}= \Tedge{M_{i_w}}{L^{q_{i_w}}R}{\bullet}{N_{i_w}}$ or $p_{i_w}= \Tedge{M_{i_w}}{R^{q_{i_w}}}{\bullet}{N_{i_w}}$ and $p_{\ell_w}= \Tedge{M_{\ell_w}}{{L}^{q_{\ell_w}}R}{\bullet}{N_{\ell_w}}$ or $p_{\ell_w}= \Tedge{M_{\ell_w}}{R^{q_{\ell_w}}}{\bullet}{N_{\ell_w}}$ for all $w$. It means that all the input words of the edges in the closed walk $\C $ are either $L^{q_{i_w}}R$ or $R^{q_{i_w}}$ for some $i_w$. Moreover, the input word of the closed walk $\C$ includes at least one run of $L$'s and one run of $R$'s. Therefore, at least one of the input words is $L^{q_{i_w}}R$ for some $i_w \geq 1$.
By \Cref{le:hrany_v_cyklu}, this is in contradiction with the fact that $\C$ is a closed walk. \qedhere
\end{enumerate}
\end{enumerate}
\end{enumerate}
\end{proof}

\subsection{The upper bound and the proof of \Cref{thm:main}}




\newcommand\Lf[1]{\widehat{#1}}
\newcommand\Lfz[1]{\Lf{#1}_0}

We need one more claim, which is a corollary of a well-known result due to the Fine and Wilf \cite{fine_wilf}.

\begin{theorem}[Fine and Wilf's theorem] \label{the:fine_wilf}
If a word $V$ has periods $p$ and $q$ and has length at least $p+q - \gcd(p,q)$, then $V$ has also period $\gcd(p,q)$.
\end{theorem}

We use this theorem in the following form.

\begin{corollary} \label{le:mocniny_slovo}
Let $\widehat{V_0} ,V_0 \in \Dj$, $V_0 \neq V_1^k$ for some $V_1 \in \Dj$ and $k \geq 2$. If
\[
\widehat{V_0}^i = V_0^j
\]
for some $i,j \geq 1$ then $\frac{j}{i} \in \N$.
\end{corollary}
\begin{proof}
In this proof, $|V|$ denotes the length of the word $V \in \Dj$.

Let $V = \widehat{V_0}^i = V_0^j$, $m = |V|$ and $p = \frac{m}{i} = |\widehat{V_0}|, q = \frac{m}{j} = |V_0|$. Therefore, the word $V$ has periods $p$ and $q$. Moreover, let $\ell = \gcd(p,q)$ and $p = \ell \widehat{p}$, $q = \ell \widehat{q}$.
Thus, $m \geq \ell \widehat{p} \widehat{q} \geq \ell (\widehat{p} + \widehat{q} - 1) = p + q - \ell$.
It means that by \Cref{the:fine_wilf}, the word $V$ has also period $\ell$ and because $V_0 \neq V_1^k$, we have $q = \ell$.

Together we obtain the following equation.
\[
\frac{j}{i} = \frac{\frac{m}{q}}{\frac{m}{p}} = \frac{p}{q} = \frac{\ell \widehat{p}}{\ell} = \widehat{p} \in \N. \qedhere
\]
\end{proof}

We recall that we want to find an estimate on the number $\sigmaC(W)$, where $W$ is the output word of the closed walk $\C = \Twalk{M}{V^\gamma}{W}{M}$ with $V$ primitive and
\begin{equation} \label{eq:C_ub_assumption}
\C \neq \left( \Twalk{M}{V^{\delta_2}}{W_2}{M} \right)^{m_2} \quad \text{ for some } m_2 \geq 2, \delta_2 \in \Z^+ \text{ and } W_2 \in \Dj.
\end{equation}

\begin{remark} \label{re:jednoduchost_nafoukle}
\Cref{co:vyfukovani_ze_symetrickeho} implies that there exists a closed walk $\widehat{\C} = \Twalk{\widehat{M}}{\widehat{V}^\gamma}{\widehat{W}}{\widehat{M}}$ where $\widehat{V} \in \tau_\kappa(V)$ and $\sigmaC(\widehat{W}) \geq \sigmaC(W)$.
If $\widehat{\C}$ can be decomposed into $m_1 \geq 2$ closed walks with the input word $\widehat{V}^{\delta_1}$, then by \eqref{eq:C_ub_assumption} and $\delta_1 m_1 = \gamma \geq 2$ we obtain a contradiction with \Cref{co:vyfukovani_ze_symetrickeho} \Cref{it:1)vyfukovani}.
Thus,
\begin{equation} \label{eq:prostota_nafoukleho_cyklu}
\widehat{\C} \neq \left (\Twalk{\widehat{M}}{\widehat{V}^{\delta_1}}{W_1}{\widehat{M}} \right)^{m_1}
\end{equation}
 for some $\delta_1 \in \Z^+, m_1 \geq 2$ and $W_1 \in \Dj$.
Further, let $m \geq 1$ be the largest possible such that
\[
\widehat{\C} = \left( \Twalk{\widehat{M}}{\widehat{V}_0^{\delta}}{\widehat{W}_0}{\widehat{M}} \right)^{m}
\]
for some $\delta \in \Z^+, \widehat{W}_0, \widehat{V}_0 \in \Dj$, and $\widehat{V}_0$ primitive.
It can happen that $m \neq 1$ or $\gamma \neq \delta$ but in all cases we have $\widehat{V}_0^{\delta m} = \widehat{V}^{\gamma}$.
By \Cref{le:mocniny_slovo}, we have $\frac{\delta m}{\gamma} \in \N$ and by \eqref{eq:prostota_nafoukleho_cyklu}, the numbers $m$ and $\gamma$ are coprime, which means that $\frac{\delta}{\gamma} \in \N$. Therefore,

\begin{equation} \label{eq:prepocty_cyklu}
m \sigmaC(\widehat{V}_0) = \frac{\gamma}{\delta} \sigmaC(\widehat{V}) \leq  \sigmaC(\widehat{V}) = \sigmaC(V) \quad  \text{ and } \quad \sigmaC(W) \leq {\sigmaC(\widehat{W})} = m\sigmaC(\widehat{W}_0).
\end{equation}

Therefore, it is sufficient to make an estimate on $\sigmaC(\Lfz W)$ only for the closed walks $\widehat{\C} = \Twalk{\widehat{M}}{{\Lfz V}^\delta}{\Lfz W}{\widehat{M}}$ where $\Lfz{V} \in \tau_\kappa(\Lfz{V})$.
\end{remark}

\begin{theorem}
Let $\Lfz \C = \Twalk{\Lf M}{\Lfz V^\delta}{\Lfz W}{\Lf M}$ with $\Lfz V$ primitive, $\Lfz V \in \tau_\kappa(\Lfz V)$, and
\[
\Lfz \C \neq \left( \Twalk{\Lf M}{\Lfz V^{\delta_1}}{W_1}{\Lf M} \right)^{m_1} \quad \text{ for some } m_1 \geq 2, \delta_1 \in \Z^+ \text{ and } W_1 \in \Dj.
\]
We have
\begin{equation} \label{eq:odhad_sigma}
\sigmaC(\Lfz W) \leq \sigmaC(\Lfz V) \sum_{M \in \LE{}} \sum_{i=\LclassMax{M}}^{2 \LclassMax{M}-1} \left(  \sigma(W_{L,M,i}) - 1 \right).
\end{equation}
($W_{L,M,i}$ is given by \Cref{le:LS_n_do_RS_n}.)
\end{theorem}

\begin{proof}
We are interested in the value of $\sigmaC(\Lfz W)$.
Since $\Lfz \C$ is a closed walk, we can choose any its vertex to be the starting vertex of the closed walk, without changing $\sigmaC(\Lfz W)$.
In other words, we may add some assumptions on $\Lf M$, to keep the notation simple.

Since $\Lfz V \in \tau_\kappa(\Lfz V)$, while repeatedly reading $\Lfz V$ and looping on $\Lfz \C$, all the runs we read are of length at least $4n$ and by \Cref{le:max_n_jednoho_pismene,le:spadnu_do_LS,le:LS_nacteni_n}, we are sure to pass through some state from $\LE{}$ at least tree times while reading one run of $L$'s.
We select this vertex as the starting vertex of $\C$, that is $M_0 = \Lf M \in \LE{}$.
Moreover, we may assume that $\C$ starts when we encounter $M_0$ for next-to-last time while reading the current run of $L$'s.
In other words, $L^{i_0}R$ is a prefix of $\Lfz V$ for some $i_0\in \{\LclassMax{M_0}, \LclassMax{M_0} + 1, \dots, 2 \LclassMax{M_0}-1\}$.

By \Cref{le:LS_n_do_RS_n}, the walk $\Lfz \C$ starts with
\[
\Twalk{M_0}{L^{i_0}R^{j_{L,M_0,i_0}}}{W_{L,M_0,i_0}}{N_0},
\]
where $N_0 = N_{L,M_0,i_0} \in \RE{}$.
The next walk on $\Lfz \C$ is
\[
\Twalk{N_0}{R^{q'_0\RclassMax{N_0}}}{R^{q'_0t'_0}}{N_0}
\]
where $t'_0 >0, q'_0 \geq 0$, $t_0'$ is such that $\Twalk{N_0}{\RclassMax{N_0}}{t_0'}{N_0}$ and $q'_0$ is chosen such that after taking this walk, the input word starts with $R^{i'_0}L$ where  $i'_0 \in \{\RclassMax{N_0}, \RclassMax{N_0} + 1, \dots, 2 \RclassMax{N_0}-1\}$.
Note that the case $q'_0 = 0$ is also possible since we may have $\RclassMax{N_0} = n$.

The symmetric version of \Cref{le:LS_n_do_RS_n} implies that the next walk that we take on $\Lfz \C$ is
\[
\Twalk{N_0}{R^{i'_0}L^{j_{R,N_0,i'_0}}}{W_{R,N_0,i'_0}}{M_1},
\]
followed by
\[
\Twalk{M_1}{L^{q_1\LclassMax{M_1}}}{L^{q_1t_1}}{M_1}.
\]
We continue to decompose $\Lfz \C$ in this manner.
This decomposition allows us to identify the output word of $\C$.
Indeed, if we put $\alpha = \frac{\sigmaC(\Lfz V)}{2}$ and recall that the input word is $\Lfz V^\delta$, then we have
\begin{equation} \label{eq:decomposition_of_C}
\Lfz W = W_{L,M_0,i_0} R^{q'_0t'_0} W_{R,N_0,i'_0}L^{q_1t_1} \cdots W_{L,M_{\alpha \delta - 1},i_{\alpha \delta - 1}} R^{q'_{\alpha \delta - 1}t'_{\alpha \delta - 1}} W_{R,N_{\alpha \delta-1},i'_{\alpha \delta - 1}}L^{q_0t_0}.
\end{equation}
By \Cref{le:LS_n_do_RS_n}, each $W_{L,M_k,i_k}$ starts with $L$ and ends with $R$, and by the symmetric version of \Cref{le:LS_n_do_RS_n}, each $W_{R,N_k,i'_k}$ starts with $R$ and ends with $L$.
Therefore, we obtain
\begin{equation} \label{eq:sigma_W_decomposition}
\sigma(\Lfz W) = 1 + \sum_{k = 0}^{\alpha\delta - 1} (\sigma(W_{L,M_k,i_k}) -1  + \sigma(W_{R,N_k,i'_k}) - 1) = \sigmaC(\Lfz W)+1.
\end{equation}

The walk $\Lfz \C$ inputs $\delta$ times the word $\Lfz V$.
We shall now focus on what can happen with a specific run in $\Lfz V$ during those $\delta$ times it is read.
Namely, let $\Lfz V = P_1 L^t P_2$, $P_1,P_2 \in \Dj$ with integer $t$ maximal possible, i.e., $L^t$ is a whole run of $L$'s in $\Lfz V$.
Let $\ell \in \{1,\ldots,\delta\}$ and $k_\ell$ be the integer such that the $\ell$-th reading of the specific run $L^t$ is associated (ends on it) with the walk from $M_{k_\ell}$ to $N_{k_\ell}$ in the above decomposition of $\Lfz \C$.
(We have $k_\ell = k_1 + \alpha (\ell-1)$.)

Assume that for $\ell \neq \ell'$ we have $M_{k_\ell} = M_{k_{\ell'}}$ and $i_{k_\ell} = i_{k_{\ell'}} $.
\Cref{le:LS_n_do_RS_n} implies that $N_{k_\ell} = N_{k_{\ell'}}$ and $j_{L,M_{k_\ell},i_{k_\ell}} = j_{L,M_{k_{\ell'}},i_{k_{\ell'}}}$.
Therefore, after the $\ell$-th and $\ell'$-th reading of the run $L^t$ we stumble upon the same state $N_{k_\ell}$ with the same input word.
This contradicts the assumptions on $\Lfz \C$.

As a consequence, we obtain an upper bound on $\delta$ by enumerating all possibilities on $M_{k_\ell}$ and $i_{k_\ell}$, where $i_{k_\ell} \in \{\LclassMax{M_{k_\ell}}, \LclassMax{M_{k_\ell}} + 1, \dots, 2 \LclassMax{M_{k_\ell}}-1\}$.
In the case that we select to focus on a run of $R$'s, \Cref{prop:assoc_sym} implies that $W_{L,M,i} = \assoc{W_{R,\assoc{M},i}}$ and therefore the resulting upper bound is the same.
Similarly, if we focus on the first run of $L$'s in $V$, which is split in two parts, the very same idea of estimate applies.
Overall, we conclude that
\[
\delta \leq \sum_{M \in \LE{}} \LclassMax{M} = \sum_{N \in \RE{}} \RclassMax{N}.
\]
and that the maximum contribution to $\sigmaC(\Lfz W)$ in \eqref{eq:sigma_W_decomposition} of the $\delta$ reads of one run equals
\[
\sum_{M \in \LE{}} \sum_{i=\LclassMax{M}}^{2 \LclassMax{M}-1} \left(  \sigma(W_{L,M,i}) - 1 \right).
\]
By the symmetry of $R$ and $L$, we need not care if the run is a run of $R$'s or $L$'s as the last number equals $\sum_{N \in \LE{}} \sum_{i=\RclassMax{N}}^{2 \RclassMax{N}-1} \left(  \sigma(W_{R,N,i}) - 1 \right)$.
Using this estimate for all the runs, we finally obtain \eqref{eq:odhad_sigma}.
\end{proof}

\begin{corollary}
If $\Lfz V = V_1 \assoc{V_1}$ for some $V_1 \in \Dj$, then either
\begin{enumerate}[(a)]
\item \label{it:coro_half_sym}
the walk $\Lfz \C$ is symmetric
or
\item \label{it:coro_half_half}
\begin{equation} \label{eq:odhad_sigma_pul}
\sigmaC(\Lfz W) \leq \frac{1}{2} \left( \sigmaC(\Lfz V) \sum_{M \in \LE{}} \sum_{i=\LclassMax{M}}^{2 \LclassMax{M}-1} \left(  \sigma(W_{L,M,i}) - 1 \right) \right).
\end{equation}
\end{enumerate}
\end{corollary}

\begin{proof}
If the walk $\Lfz \C$ is not symmetric, then, in the estimate \eqref{eq:odhad_sigma} we do not need to count the symmetric possibilities in the following sense.
If when reading a run of $L$'s in $V$, we count the state $M \in \LE$ with $i \in \left\{ \LclassMax{M}, \ldots, 2 \LclassMax{M}-1 \right \}$, then when reading the symmetric run of $R$'s we cannot pass through $\assoc{M} \in \RE$ associated with the integer $i$.
The symmetric run of $R$'s exists due to the assumption $\Lfz V = V_1 \assoc{V_1}$.
Thus, we can count only half of all the possible states $(M_{k_\ell},i_{k_\ell})$, resp. $(N_{k_\ell},i_{k_\ell}')$, which gives the estimate in \cref{it:coro_half_half}.
\end{proof}

We now transform the last corollary into the terms of the period of the continued fraction of $x$ after the transformation.

\begin{theorem} \label{the:odhad_pi}
Let $x$ be a quadratic irrational number and $N \in \D{}$.
We have
\[
\per(h_N(x)) \leq \per(x)\sum_{M \in \LE{}} \sum_{i=\LclassMax{M}}^{2 \LclassMax{M}-1} \left(  \sigma(W_{L,M,i}) - 1 \right).
\]
\end{theorem}

\begin{proof}
Let $V$ be the repetend of the LR-representation of $x$.
It implies that $V$ is primitive.
Using~\Cref{thm:dostanu_se_do_Dn}, we may assume that the calculation of the tail of $h_N(x)$ is given by a closed walk $\C = \Twalk{M}{V^{\gamma}}{W}{M}$ satisfying \eqref{eq:C_ub_assumption}.

We find the closed walk $\widehat{\C}_0 = \Twalk{\widehat{M}}{\widehat{V}_0^{\delta}}{\widehat{W}_0}{\widehat{M}}$ with $\widehat{V}_0^{\delta m} \in \tau_\kappa(V^\gamma)$, $\widehat{V_0}$ primitive and $\widehat{M} \in \DB{}$ given by \Cref{co:vyfukovani_ze_symetrickeho} and \Cref{re:jednoduchost_nafoukle}, using the notation therein.
According to \eqref{eq:prepocty_cyklu}, we have
\begin{equation} \label{eq:pf:nafuk}
\sigmaC(W) \leq m \sigmaC(\widehat{W}_0),
\end{equation}
 \begin{equation} \label{eq:pf:V}
m \sigmaC(\widehat{V}_0) \leq \sigmaC(V),
\end{equation}

from \Cref{le:vypocet_per} it follows that
    \begin{equation} \label{eq:pf:A2}
      \per(h_N(x)) \leq \sigmaC(W)
    \end{equation}

and that
 \begin{equation} \label{eq:pf:per_sigma}
 \frac{1}{2} \sigmaC(V) \leq \per(x).
 \end{equation}

Set $S = \sum_{M \in \LE{}} \sum_{i=\LclassMax{M}}^{2 \LclassMax{M}-1} \left(  \sigma(W_{L,M,i}) - 1 \right)$.

We split the proof into several cases.

\begin{enumerate}[label=(\textbf{\Alph*})]
  \item $\widehat{\C}_0$ is symmetric.

  $\widehat{V}_0 = V_1\assoc{V_1}$ for some $V_1 \in \Dj$.

  \begin{enumerate}[label=(\textbf{\Alph{enumi}}\arabic*)] 
    \item $W = W_1\assoc{W_1}$ for some $W_1 \in \Dj$.

  By \Cref{le:vypocet_per} we have $\per(h_N(x)) \leq \frac{\sigmaC(W)}{2}$. Therefore:

    \[
    \per(h_N(x)) \leq \frac{\sigmaC(W)}{2} \overset{\eqref{eq:pf:nafuk}}{\leq} \frac{1}{2} m\sigmaC(\widehat{W}_0) \overset{\eqref{eq:odhad_sigma}}{\leq} \frac{1}{2} m \sigmaC(\widehat{V}_0) S \overset{\eqref{eq:pf:V}}{\leq}\frac{1}{2} \sigmaC(V)S \overset{\eqref{eq:pf:per_sigma}}{\leq} \per(x)S.
    \]

    \item $W \neq W_1\assoc{W_1}$ for all $W_1 \in \Dj$

    \begin{enumerate}[label=(\textbf{\Alph{enumi}}\arabic{enumii}\roman*)] 
      \item $V = V_2\assoc{V_2}$ for some $V_2 \in \Dj$

      By \Cref{co:vyfukovani_ze_symetrickeho} \Cref{it:2)vyfukovani}, we obtain a contradiction with \Cref{le:sym_dvojstav}.

      \item $V \neq V_2\assoc{V_2}$ for all $V_2 \in \Dj$.

      By \Cref{le:vypocet_per} we have $\sigmaC(V) = \per(x)$. Therefore:

      \[
      \per(h_N(x)) \overset{\eqref{eq:pf:A2}}{\leq} \sigmaC(W) \overset{\eqref{eq:pf:nafuk}}{\leq} m \sigmaC(\widehat{W}_0) \overset{\eqref{eq:odhad_sigma}}{\leq} m \sigmaC(\widehat{V}_0) S \overset{\eqref{eq:pf:V}}{\leq} \sigmaC(V) S \overset{}{=} \per(x)S.
      \]
    \end{enumerate}
  \end{enumerate}
  \item $\widehat{\C}_0$ is not symmetric

  \begin{enumerate}[label=(\textbf{\Alph{enumi}}\arabic*)] 
    \item $\widehat{V}_0 = V_1\assoc{V_1}$ for some $V_1 \in \Dj$. Therefore:

    \[
      \per(h_N(x)) \overset{\eqref{eq:pf:A2}}{\leq} \sigmaC(W) \overset{\eqref{eq:pf:nafuk}}{\leq} m\sigmaC(\widehat{W}_0) \overset{\eqref{eq:odhad_sigma_pul}}{\leq} \frac{1}{2} m \sigmaC(\widehat{V}_0) S \overset{\eqref{eq:pf:V}}{\leq} \frac{1}{2} \sigmaC(V) S \overset{\eqref{eq:pf:per_sigma}}{\leq} \per(x)S.
    \]

    \item $\widehat{V}_0 \neq V_1\assoc{V_1}$ for all $V_1 \in \Dj$

    The fact that $\widehat{V}_0 \neq V_1\assoc{V_1}$ implies $V \neq V_2\assoc{V_2}$ for all $V_2 \in \Dj$. By \Cref{le:vypocet_per} we have $\sigmaC(V) = \per(x)$. Therefore:

    \[
      \per(h_N(x)) \overset{\eqref{eq:pf:A2}}{\leq} \sigmaC(W) \overset{\eqref{eq:pf:nafuk}}{\leq} m\sigmaC(\widehat{W}_0) \overset{\eqref{eq:odhad_sigma}}{\leq} m \sigmaC(\widehat{V}_0) S  \overset{\eqref{eq:pf:V}}{\leq} \sigmaC(V) S \overset{}{=} \per(x)S. \qedhere
    \]

  \end{enumerate}
\end{enumerate}
\end{proof}

As the last theorem holds for all quadratic irrational numbers $x$ and all matrices $N \in \D{}$, the proof of \Cref{thm:main} follows.

\begin{proof}[Proof of \Cref{thm:main}]
Trivially, we may assume that $N \in \D{}$.

We shall first prove the upper bound on $\per(h_N(x))$ which follows directly from \Cref{the:odhad_pi} with
\[
S_n =  \sum_{M \in \LE{}} \sum_{i=\LclassMax{M}}^{2 \LclassMax{M}-1} \left(  \sigma(W_{L,M,i}) - 1 \right).
\]

By \Cref{def:LE_n}, we have $M \in \LE{} \iff M = M_{t,u}= \mat{t}{0}{u}{m}$ where $mt = n$ and if we put $g_t = \gcd(m,t) = \gcd(t,\frac{n}{t})$, $I_{t} = \begin{cases} \{0\} & \text{ for } g_t = 1 \\ \{1, \dots, g_t - 1\} & \text{ otherwise}\end{cases}$, then $u \in I_{t}$.
It follows from \Cref{le:LS_n_do_RS_n} that we have
\begin{equation} \label{eq:prepocet_sumy}
 S_n
 =
 \sum_{\substack{ t \in \N \\ t \mid n}} \sum_{u \in I_{t}} \sum_{i=\LclassMax{M_{t,u}}}^{2 \LclassMax{M_{t,u}}-1} \left( 2\left\lfloor \frac{\xi(im+u, t)}{2} \right\rfloor + 1 \right).
\end{equation}

We rearrange the two last sums into one.
By \Cref{le:char_toceni_na_L}, $\LclassMax{M_{t,u}}$ is the least positive integer such that
\begin{equation} \label{eq:ht}
m \nu_L(M_{t,u}) = ht
\end{equation}
for some $h \in \N \setminus \{0\}$ and therefore $\nu_L(M_{t,u}) = \frac{t}{g_t}$. Therefore, we have
$i_1m+u \not \equiv i_2m+u \pmod{t}$ for all $i_1, i_2 \in \{\nu_L(M_{t,u}),\nu_L(M_{t,u})+1,\dots,2\nu_L(M_{t,u})-1\}, i_1 \neq i_2$.
It means that $\{im+u \pmod{t} \colon  i \in \{\nu_L(M_{t,u}),\nu_L(M_{t,u})+1,\dots,2\nu_L(M_{t,u})-1\}\} = \{k \colon k \in \N,k<t,k \equiv u \pmod{g_t}\}$ and $\# \{k \colon k \in \N,k<t,k \equiv u \pmod{g_t}\} = \nu_L(M_{t,u})$.

Let $J_t  = \begin{cases}  \emptyset & \text{for } g_t = 1, \\ \{i g_t \colon i \in \N\} & \text{ otherwise.} \end{cases}$.
Together with the facts that $u = 0 $ for $g_t = 1$ and $u \in \{1, \dots, g_t - 1\}$ otherwise we have
\[
\{im+u \! \! \pmod{t} \colon  i \in \{\nu_L(M_{t,u}),\nu_L(M_{t,u})+1,\dots,2\nu_L(M_{t,u})-1\}, u \in I_t\} =
\{k \colon k \in \N,k<t, k \not \in J_t\}
\]
and $\# \{k \colon k \in \N,k<t, k \not \in J_t\} = \sum_{u \in I_{t}}\nu_L(M_{t,u})$.
Now it remains to realize that for all $i \geq \nu_L(M_{t,u})$ we have by \eqref{eq:ht} that $im + u \geq th+u \geq t$ and by definition of $\xi$, we have $\xi(k,t) = \xi(at+k,t)$ for all $a \in \N$ and $k\geq t$. We conclude
\[
\sum_{u \in I_{t}} \sum_{i=\LclassMax{M_{t,u}}}^{2 \LclassMax{M_{t,u}}-1} \left( 2\left\lfloor \frac{\xi(im+u, t)}{2} \right\rfloor + 1 \right) = \sum_{\substack{j = t \\ j \not \in J_t }}^{2t-1} \left (2\left\lfloor \frac{\xi(j,t)}{2} \right\rfloor +1 \right)
\]
which together with \eqref{eq:prepocet_sumy} proves the upper bound.

As the inverse M\"obius transformation preserves the determinant of its associated matrix, i.e, $h_N^{-1} = h_{N'}$ for some $N' \in \D{}$, the lower bound follows from the upper bound.
\end{proof}

\section{Concluding remarks and experiment results} \label{sec:conclusion}

We have tested the obtained upper bound of \Cref{thm:main} for various values of $n$.
Some of the results can be seen in \Cref{tab:prodlouzeni}.
The experiments indicate that the upper bound is sharp for $n=2$, all prime $n$ with $n \equiv 3 \pmod 4$ and for some composite numbers (for example $n = 9,14,27$).
For some other values of $n$ the difference between our estimate $S_n$ and the experimentally obtained factor of prolongation (denoted $S_n(x)$ in the table) can be relatively large (for example for $n = 18$).

{
\renewcommand*{\arraystretch}{1.2}
\begin{table}[!htb]
\centering
 \begin{tabular}{c|c|c|c}
 $n$ & $S_n$ & $S_n(x)$ & $x$  \\
 \hline
7 & 24 & 24.0 & $[\overline{4390}]$ \\
8 & 36 & 26.0 & $[\overline{4792, 4423}]$ \\
9 & 36 & 36.0 & $[\overline{4696}]$ \\
13 & 52 & 51.0 & $[\overline{4771, 4930}]$ \\
14 & 80 & 80.0 & $[\overline{4693}]$ \\
15 & 76 & 67.2 & $[2904, 189, \overline{4662, 4147, 4872, 4669, 4875}]$ \\
18 & 120 & 68.0 & $[\overline{4908, 4057}]$ \\
20 & 120 & 104.4 & $[4495, 520, \overline{4803, 4060, 4805, 4930, 4643}]$ \\
24 & 164 & 90.0 & $[\overline{4380, 4843}]$ \\
27 & 144 & 144.0 & $[\overline{4384}]$ \\
81 & 538 & 532.0 & $[\overline{4232}]$ \\
\end{tabular}
  \caption{Experimental results on the bound of \Cref{thm:main}. $S_n(x)$ denotes an experimental lower bound on $\sup \left\{ \frac{ \per(h_M(x))}{ \per(x)} \colon M \in \D{} \right\}$.}
  \label{tab:prodlouzeni}
\end{table}
}

The difference between $S_n$ and $S_n(x)$ is caused by the fact that in some cases the closed walk $\C$ cannot go through all of the transitions that we have considered in the estimate \eqref{eq:odhad_sigma}.
For composite numbers, a sharper estimate depends on the value of $n$ and its divisors.
If $n$ is prime, the sharp bound may be proven to be
\begin{equation} \label{eq:n_prime}
\begin{aligned}
\sup \left\{ \frac{ \per(h_M(x))}{ \per(x)} \colon M \in \D{} \text{ and } x \text{ is quadratic irrational} \right\} = \\
= \begin{cases} 5 & \text{ if } n = 2, \\
\displaystyle 2 + 2\sum_{i=1}^{\frac{n-1}{2}} (\xi(i,n) + 2)  & \text{ if } n \equiv 3 \pmod{4}, \\
\displaystyle 1 + 2\sum_{i=1}^{\frac{n-1}{2}} (\xi(i,n) + 2)  & \text{ if } n \equiv 1 \pmod{4}.
\end{cases}
\end{aligned}
\end{equation}
For $n \equiv 3 \pmod{4}$ the bound in fact equals $S_n$, the formula is only simplified.
The bound for $n \equiv 3 \pmod{4}$ equals $S_n-1$, corresponding to the case when $\C$ cannot pass through all possible vertices.
A corresponding experiment is for $n=13$ in \Cref{tab:prodlouzeni}.

We do not give a proof of the formula \eqref{eq:n_prime} as it is only for a very special case and requires some more technical claims.

We do not provide experiment results on the lower bound as the behaviour is completely analogous, one only needs to consider the inverse of the given M\"obius transformation.

\section*{Acknowledgements}

The work was supported by the Ministry of Education, Youth and Sports of the Czech Republic, project no. CZ.02.1.01/0.0/0.0/16\_019/0000778.
H. Ř. acknowledges support by the Grant Agency of the Czech Technical University in Prague, grant No. SGS17/193/OHK4/3T/14.
The computer experiments were done using the computer algebra system SageMath \cite{sage_2018}.

\bibliographystyle{siam}
\IfFileExists{biblio.bib}{\bibliography{biblio}}{\bibliography{../biblio/biblio}}

\end{document}